\documentclass[11pt]{amsart}

\usepackage{amsfonts}
\usepackage{amssymb}
\usepackage{amsthm}
\usepackage{amsmath}
\usepackage{hyperref}
\usepackage{enumitem}
\usepackage{tikz}
\usepackage{pgf}
\usepackage{graphicx}
\usepackage[top=3.5cm, bottom=3.5cm, left=2.5cm, right=2.5cm]{geometry}
\usepackage[capitalise]{cleveref}

\newtheorem{prop}{Proposition}[section]
\newtheorem{theorem}[prop]{Theorem}
\newtheorem{lemma}[prop]{Lemma}

\newtheorem{corollary}[prop]{Corollary}

\newtheorem*{prop*}{Proposition}

\theoremstyle{definition}

\theoremstyle{remark}

\newtheorem{example}[prop]{Example}

\newtheorem*{remark*}{Remark}
\newtheorem{remark}[prop]{Remark}

\theoremstyle{theorem}

\newcommand{\R}{\mathbb{R}}

\newcommand{\A}{\mathbb{A}}
\newcommand{\N}{\mathbb{N}}

\newcommand{\Z}{\mathbb{Z}}

\newcommand{\E}{\mathbb{E}}
\newcommand{\EE}{\mathcal{E}}

\newcommand{\fq}{{\lfloor q\rfloor}}

\newcommand{\fp}{{\lfloor p\rfloor}}

\newcommand{\Prob}{\mathbb{P}}

\newcommand{\eps}{\varepsilon}
\newcommand{\hit}{\textup{hit}}
\newcommand{\di}{\textup{div}}

\newcommand{\ca}{\text{\textup{cap}}}

\newcommand{\Aut}{\text{\textup{Aut}}\,}
\newcommand{\Cay}{\mathcal{C}}

\newcommand*{\diam}{\mathop{\textup{diam}}\nolimits}
\newcommand*{\rad}{\mathop{\textup{rad}}\nolimits}

\newcommand*\diff{\mathop{}\!\mathrm{d}}

\numberwithin{equation}{section}
\numberwithin{table}{section}

\title[Volume growth, isoperimetry and escape probability in vertex-transitive graphs]{Sharp relations between volume growth, isoperimetry and escape probability in vertex-transitive graphs}
\date{}
\author{Romain Tessera}
\address{Institut de Math\'ematiques de Jussieu-Paris Rive Gauche, France}
\email{romain.tessera@imj-prg.fr}
\author{Matthew Tointon}
\address{School of Mathematics, University of Bristol, United Kingdom}
\email{m.tointon@bristol.ac.uk}
\thanks{M. Tointon was partially supported by a Junior Research Fellowship from Homerton College, Cambridge; grant FN 200021\_163417/1 of the Swiss National Fund for scientific research; the Stokes Research Fellow from Pembroke College, Cambridge; and by a travel grant from the Pembroke College Fellows' Research Fund.}

\begin{document}
\maketitle

\begin{abstract}
We prove sharp bounds on the probability that the simple random walk on a vertex-transitive graph escapes the ball of radius $r$ before returning to its starting point. In particular, this shows that if the ball of radius $r$ has size slightly greater than quadratic in $r$ then this probability is bounded from below. On the other hand, we show that if the ball of radius $r$ has volume slightly less than cubic in $r$ then this probability decays logarithmically for all larger balls. These results represent a finitary refinement of Varopoulos's theorem that a random walk on a vertex-transitive graph is recurrent if and only if the graph has at most quadratic volume growth. They also imply the existence of a \emph{gap} at $0$ for escape probabilities: there exists a universal constant $c>0$ such that the random walk on an arbitrary vertex-transitive graph is either recurrent or has a probability of at least $c$ of escaping to infinity. We also prove versions of these results for finite graphs, in particular confirming and strengthening a conjecture of Benjamini and Kozma from 2002. Amongst other things, we also generalise our results to give a sharp finitary version of the characterisation of $p$-parabolic vertex-transitive graphs, prove a number of sharp isoperimetric inequalities for vertex-transitive graphs, and prove a locality result for the escape probability of the random walk on a vertex-transitive graph that can be seen as an analogue of Schramm's locality conjecture for the critical percolation probability.
\end{abstract}
\tableofcontents

\section{Introduction}

One of the oldest results about random walks is P\'olya's theorem that the random walk on the grid $\Z^d$ is \emph{recurrent} if $d\le2$, and \emph{transient} otherwise -- that is to say, if $d\le2$ then the random walk will almost surely return to $0$ infinitely often, whereas if $d\ge 3$ then it will almost surely visit $0$ only finitely many times. Varopoulos extended this result to arbitrary connected vertex-transitive graphs, showing that every such graph with super-quadratic volume growth has a transient random walk. The main purpose of this paper is to prove quantitative, finitary analogues of Varopoulos's theorem. In particular, we prove an analogue for finite graphs that completely resolves a conjecture of Benjamini and Kozma \cite{bk}. Our arguments build on the results and techniques of a number of our previous papers \cite{proper.progs,tt.trof,ttBalls}, as well as a celebrated result of Breuillard, Green and Tao on the structure of so-called \emph{approximate groups} \cite{bgt}.

A graph $\Gamma$ is called \emph{vertex-transitive} if its automorphism group $\Aut(\Gamma)$ acts transitively on the set of vertices of $\Gamma$. A particular class of vertex-transitive graphs is the class of \emph{Cayley graphs}. Given a group $G$ with a symmetric generating set $S$, we define the \emph{Cayley graph $\Cay(G,S)$ of $G$ with respect to $S$} to be the graph whose vertices are the elements of $G$ with an edge between $x$ and $y$ precisely when and there exists $s\in S\setminus\{1\}$ such that $xs=y$.

If $\Gamma$ is a graph then we abuse notation slightly by also writing $\Gamma$ for the set of vertices of $\Gamma$. We write $E(\Gamma)$ for the set of edges. In this paper we will generally assume that $\Gamma$ is connected. Define the \emph{degree} of $\Gamma$ to be $\deg(\Gamma)=\sup_{x\in\Gamma}\deg(x)$. Write $d=d_\Gamma$ for the graph metric on the vertices of $\Gamma$. Given $x\in\Gamma$ and $r\in\N_0$, write $B_\Gamma(x,r)=\{u\in\Gamma:d(u,x)\le r\}$ for the \emph{ball} of radius $r$ centred at $x$, and $S_\Gamma(x,r)=\{u\in\Gamma:d(u,x)=r\}$ for the \emph{sphere} of radius $r$ centred at $x$. Abbreviate $\beta_\Gamma(x,r)=|B_\Gamma(x,r)|$ and $\sigma_\Gamma(x,r)=|S_\Gamma(x,r)|$. If $\Gamma$ is vertex transitive then $\beta_\Gamma(x,r)$ and $\sigma_\Gamma(x,r)$ are independent of $x$, and so we abbreviate them further by $\beta_\Gamma(r)$ and $\sigma_\Gamma(r)$, respectively.

Here, as elsewhere in this series of papers, the notation $A\ll B$ and $B\gg A$ both mean that there exists an absolute constant $C>0$ such that $A\le CB$, and the notation $A\asymp B$ means that $A\ll B\ll A$.

We adopt the standard convention that, given a vertex $x$ of a graph $\Gamma$, the notation $\Prob_x$ represents probability conditioned on the simple random walk starting at $x$. Given a set of vertices $Y\subset\Gamma$, write $T_Y$ for the \emph{hitting time} of $Y$, which is to say the earliest time at which the random walk lies in $Y$, with $T_Y=\infty$ if the random walk never lies in $Y$. Write $T_Y^+$ for the earliest time after zero that the random walk lies in $Y$; this is different from $T_Y$ only if the random walk starts in $Y$. If $Y$ is a singleton $\{y\}$, we generally write $T_y$ or $T_y^+$ instead of $T_{\{y\}}$ or $T_{\{y\}}^+$. Given in addition $x\in\Gamma\setminus Y$, we denote by $\Prob[\,x\to Y\,]$ the probability
\[
\Prob[\,x\to Y\,]=\Prob_x[\,T_Y<T_x^+\,]
\]
that the random walk starting at $x$ hits $Y$ without first returning to $x$. By convention we set $\Prob[\,x\to\varnothing\,]=0$.

For every $x\in\Gamma$, the events $[\,T_{S_\Gamma(x,r)}<T_x^+\,]$ are decreasing in $r$, so we may define $\Prob[\,x\to\infty\,]$ to be the limit
\[
\Prob[\,x\to\infty\,]=\lim_{r\to\infty}\Prob[\,x\to S_\Gamma(x,r)\,].
\]
Thus $\Prob[\,x\to\infty\,]$ is the probability that the random walk, starting at $x$, never returns to $x$. This is sometimes called the \emph{escape probability} from $x$. A connected graph $\Gamma$ is therefore transient if and only if $\Prob[\,x\to\infty\,]>0$ for some (equivalently, every) vertex $x\in\Gamma$ (see e.g. \cite[Theorem 2.3]{ly-per}).

Using the above notation, Varopoulos's theorem can be stated as follows.
\begin{theorem}[Varopoulos]\label{thm:var}
Let $\Gamma$ be a connected, locally finite vertex-transitive graph, and suppose that $\beta_\Gamma(r)/r^2\to\infty$ as $r\to\infty$. Then $\Prob[\,x\to\infty\,]>0$ for every $x\in\Gamma$.
\end{theorem}
It is well known that $\beta_\Gamma(r)/r^2\to\infty$ implies the even stronger condition $\beta_\Gamma(r)\gg r^3$. Actually the proof requires that $\beta_\Gamma(r)\gg r^2\log r$, a subtlety that will become relevant in the finitary setting of the present paper.

Our first result is a quantitative finitary refinement of Theorem \ref{thm:var}, where note that the $r^2\log r$ featuring in Varopoulos's proof appears in plain view.
\begin{theorem}\label{thm:main.unimod.rw}
Let $\Gamma$ be a connected, locally finite vertex-transitive graph, let $x\in\Gamma$, and let $r\in\N$, with $r<\diam(\Gamma)$ if $\Gamma$ is finite. Then
\[
\min\left\{1, \frac{\beta_\Gamma(r)}{\deg(\Gamma)r^2\log r}\right\}\ll\Prob[\,x\to S_\Gamma(x,r+1)\,]\ll\frac{\beta_\Gamma(r)}{r^2}.
\]
\end{theorem}
Note that Theorem \ref{thm:main.unimod.rw} shows in particular that if $\beta_\Gamma(r)$ grows at least as fast as $r^2\log r$ then the random walk on $\Gamma$ is transient, recovering Theorem \ref{thm:var}, and in fact giving the following quantitative strengthening.
\begin{corollary}[gap at $0$ for the escape probability on vertex-transitive graphs]\label{cor:gap}
There exists a universal constant $c>0$ such that the random walk on every connected, locally finite vertex-transitive graph is either recurrent or has escape probability at least $c$.
\end{corollary}

In Section \ref{sec:bk.examples}, we give examples to show that Theorem \ref{thm:main.unimod.rw} is in general sharp up to the constants implied by the $\ll$ notation. Nonetheless, in the special case in which $\beta_\Gamma(r)$ is sub-quadratic in $r$ we can dispense with the logarithm in the lower bound, as follows.
\begin{theorem}\label{thm:unimod.linear.rw}
Let $\eps>0$. Let $\Gamma$ be a connected, locally finite vertex-transitive graph, let $x\in\Gamma$, and let $r\in\N$ be such that $r<\diam(\Gamma)$. Suppose that $\beta_\Gamma(r)\le r^{2-\eps}$. Then
\[
\Prob[\,x\to S_\Gamma(x,r+1)\,]\gg_\eps\frac{\beta_\Gamma(r)}{\deg(\Gamma)r^2}.
\]
\end{theorem}
We prove Theorem \ref{thm:unimod.linear.rw} in the more general form of Theorem \ref{thm:main.unimodp.linear}, below.

\subsection*{A converse to Varopoulos's theorem}
A well-known converse to Varopoulos's theorem says that any vertex-transitive graph of at most quadratic growth has a recurrent random walk. This does not follow from Theorem \ref{thm:main.unimod.rw}, since the upper bound of that theorem is merely constant in the case of quadratic growth. This might seem surprising given that the upper bound of Theorem \ref{thm:main.unimod.rw} is optimal. To reconcile these seemingly opposing statements -- that this finitary upper bound is optimal, and yet still too weak to imply its asymptotic analogue -- it is instructive to consider the following examples achieving the upper bound.
\begin{example}\label{ex:rw.quad.ub.optimal}
Let $\Gamma_r$ be the Cayley graph of $\Z\times(\Z/r^{1/2}\Z)^2$ with respect to the generating set $\{-1,0,1\}^3$. We show in Proposition \ref{prop:resist.lb.ex} that there exists a constant $\alpha>0$ such that for every $r\in\N$ we have $\Prob[\,x\to S_{\Gamma_r}(x,r)\,]\ge\alpha$ for every $x\in\Gamma_r$. In particular, the upper bound of Theorem \ref{thm:unimod.linear.rw} cannot be improved in the case $\beta_\Gamma(r)\asymp r^2$. Despite this, note that each graph $\Gamma_r$ is recurrent. Indeed, Theorem \ref{thm:unimod.linear.rw} shows that $\Prob[\,x\to S_{\Gamma_r}(x,n)\,]$ decays like $1/n$ for $n\ge r$; it is just that the radius $r$ at which Theorem \ref{thm:unimod.linear.rw} begins to detect this decay can be arbitrarily large.
\end{example}
To fill this gap in the conclusion of Theorem \ref{thm:main.unimod.rw} we provide the following finitary converse to Varopoulos's theorem.
\begin{theorem}\label{thm:var.converse.rw}
There exists $\eps>0$ such that if $\Gamma$ is a locally finite vertex-transitive graph satisfying
\begin{equation}\label{var.conv.hyp}
\beta_\Gamma(n)\le\eps n^3\beta_\Gamma(1)
\end{equation}
for some $n\in\N$ then for every $x\in\Gamma$ we have
\[
\Prob[\,x\to S_\Gamma(x,r)\,]\ll\frac{\beta_\Gamma(n)}{n^2\log(r/n)}
\]
for every $r\ge n$.
\end{theorem}
Note in particular that Theorem \ref{thm:var.converse.rw} is able to detect the recurrence of the graph $\Gamma_r$ from Example \ref{ex:rw.quad.ub.optimal} with reference only to the ball of radius $r$. We state a more precise version of Theorem \ref{thm:var.converse.rw} in Theorem \ref{thm:var.converse}, below.
\begin{corollary}\label{cor:gap.orig}
There exists a universal constant $c>0$ such that if $\Gamma$ is a connected, locally finite vertex-transitive graph satisfying
\[
\Prob[\,x\to S_\Gamma(x,n)\,]<c
\]
for some $n\in\N$ and $x\in\Gamma$ then
\[
\Prob[\,x\to S_\Gamma(x,r)\,]\ll\frac{\beta_\Gamma(n_0)}{n_0^2\log(r/n_0)}
\]
for every $r\ge n$.
\end{corollary}
\begin{proof}
Theorem \ref{thm:main.unimod.rw} implies that $\beta(n)\ll cn^2(\log n)\beta(1)$, so for small enough $c$ the desired result follows from Theorem \ref{thm:var.converse.rw}.
\end{proof}

\subsection*{Analogues of transience for finite graphs}
It is worth noting that Theorem \ref{thm:main.unimod.rw} has content for finite graphs. However, in that setting a possible analogue of $\Prob[\,x\to\infty\,]$ is $\min_{x,y\in\Gamma}\Prob[\,x\to y\,]$, and a possible analogue of transience is thus that $\min_{x,y\in\Gamma}\Prob[\,x\to y\,]$ is bounded away from zero.

A conjecture of Benjamini and Kozma proposes, amongst other things, an analogue of Varopoulos's theorem for finite vertex-transitive graphs in terms of this notion of transience \cite[Conjecture 4.2]{bk}. Denoting by $\diam(\Gamma)$ the diameter of $\Gamma$, Benjamini and Kozma conjecture in particular that $\min_{x,y\in\Gamma}\Prob[\,x\to y\,]$ should be uniformly bounded away from zero for any vertex-transitive graph $\Gamma$ of bounded degree satisfying
\begin{equation}\label{eq:bk.cond}
\diam(\Gamma)^2\ll\frac{|\Gamma|}{\log|\Gamma|}.
\end{equation}
Note that \eqref{eq:bk.cond} amounts to saying that the ball of radius $\diam(\Gamma)$ has volume at least slightly greater than quadratic in its radius.

Our next result confirms this aspect of Benjamini and Kozma's conjecture and extends it to graphs of arbitrarily large degree. (We resolve the rest of their conjecture in Corollary \ref{cor:bk.lin}, below.)
\begin{theorem}\label{thm:main.rw}
Let $\Gamma$ be a finite, connected, vertex-transitive graph. Then
\[
\min\left\{1,\frac{|\Gamma|}{\deg(\Gamma)\diam(\Gamma)^2\log(|\Gamma|/\deg(\Gamma))}\right\}\ll\min_{x,y\in\Gamma}\Prob[\,x\to y\,]\ll\frac{|\Gamma|}{\diam(\Gamma)^2}.
\]
\end{theorem}
We prove Theorem \ref{thm:main.rw} in the equivalent form of Theorem \ref{thm:main}, below. As with Theorem \ref{thm:main.unimod.rw}, we provide examples in Section \ref{sec:bk.examples} to show that the bounds of Theorem \ref{thm:main.rw} are in general both sharp up to the multiplicative constants, but in the special case in which $\diam(\Gamma)$ is comparable to $|\Gamma|$ we can dispense with the logarithm in the lower bound, as follows.
\begin{theorem}\label{thm:main.linear.rw}
Let $\eps>0$. Let $\Gamma$ be a finite connected vertex-transitive graph such that $\diam(\Gamma)\ge|\Gamma|^{\frac{1+\eps}{2}}$. Then
\[
\min_{x,y\in\Gamma}\Prob[\,x\to y\,]\gg_\eps\frac{|\Gamma|}{\deg(\Gamma)\diam(\Gamma)^2}.
\]
\end{theorem}
We prove Theorem \ref{thm:main.linear.rw} in the more general form of Theorem \ref{thm:main.linearp}, below. In Section \ref{sec:bk.examples} we provide examples to show that the bound in Theorem \ref{thm:main.linear.rw} is sharp up to the multiplicative constant; in particular, the gap of $\deg(\Gamma)$ between the lower bound of Theorem \ref{thm:main.linear.rw} and the upper bound of Theorem \ref{thm:main.rw} is unavoidable.

\subsection*{Isoperimetric inequalities} It is well known that {\it isoperimetric inequalities} give information about random walks, and indeed Varopoulos's proof of his result relied on using an isoperimetric inequality to bound the return probabilities $p_{t}(x,x)=\Prob_x(X_t=x)$.  
However, the finitary nature of Theorem 1.2 requires a different approach. To that end, adapting an argument of Benjamini and Kozma \cite{bk}, we manage to bound $\Prob[\,x\to S_\Gamma(x,r+1)\,]$ directly via isoperimetric inequalities. 

Given a set $A$ in a graph $\Gamma$, we write $\partial A$ for the \emph{external vertex boundary} $\partial A=\{x\in\Gamma\setminus A:(\exists a\in A)(a\sim x)\}$. By an \emph{isoperimetric inequality} we mean a lower bound on $|\partial A|$ in terms of $|A|$. 
The isoperimetric inequality we ultimately need in order to prove our main results is slightly complicated to state, so we leave it until Theorem \ref{thm:isoperimRel}. However, our techniques also yield two isoperimetric inequalities that are of independent interest. The first is the following inequality, which verifies and generalises another conjecture of Benjamini and Kozma \cite[Conjecture 4.1]{bk}.
\begin{theorem}\label{thm:bk.iso.orig}
Let $q\ge1$, let $\Gamma$ be a locally finite vertex-transitive graph, let $r\in\N$ be such that $r\le\diam(\Gamma)$, and suppose that $\beta_\Gamma(r)\ge r^q$. Then for every finite subset $A\subset\Gamma$ with $|A|\le\frac{1}{2}\beta_\Gamma(r)$ we have
\[
|\partial A|\gg_\fq|A|^{\frac{q-1}{q}}.
\]
\end{theorem}
Theorem \ref{thm:bk.iso.orig} is sharp; indeed, see Proposition \ref{prop:iso.conv} for a converse. We actually prove a slight refinement of Theorem \ref{thm:bk.iso.orig}, which we state as Theorem \ref{thm:bk.iso}.

Our second isoperimetric inequality is the following.
\begin{theorem}\label{thm:isoperim.orig.rel.bgt}
Let $q\ge1$, let $\Gamma$ be a locally finite  vertex-transitive graph, let $r\in\N$ be such that $r\le\diam(\Gamma)$, and suppose that $\beta_\Gamma(r)\ge r^q\beta_\Gamma(1)$. Then there exists $b(q)$ with $b(q)\to\infty$ as $q\to\infty$ such that for every subset $A\subset\Gamma$ with $|A|\le\beta_\Gamma(r)/2$ we have
\[
|\partial A|\gg_\fq\beta_\Gamma(1)^\frac{1}{b(q)}|A|^\frac{b(q)-1}{b(q)}.
\]
\end{theorem}
In Theorem \ref{thm:isoperim.orig.rel} we prove Theorem \ref{thm:isoperim.orig.rel.bgt} with an explicit and optimal function $b$. This sharpens and generalises a result of Breuillard, Green and Tao, who proved it with a non-explicit function $b$ and with $\Gamma$ assumed to be a Cayley graph \cite[Corollary 11.15]{bgt}. 

By a famous argument of Coulhon and Saloff-Coste, one can deduce isoperimetric inequalities from lower bounds on the size of balls. However, in \cref{thm:bk.iso.orig,thm:isoperim.orig.rel.bgt} the only ball on which we have a lower volume bound a priori is the ball of radius $r$. A crucial ingredient in this paper is therefore the following result of ours \cite[Corollary 1.17]{ttBalls} (see also our earlier papers \cite{proper.progs,tt.trof} for part \ref{item:growth.abs}), the contrapositive of which shows that this lower bound implies lower bounds on the sizes of all smaller balls.
\begin{theorem}[polynomial volume implies polynomial growth]\label{thm:Growth}\label{thm:tt}\label{thm:tt.rel}
For every $d\in\N$ there exists $\eps=\eps(d)>0$ such that the following holds for any connected, locally finite vertex-transitive graph $\Gamma$.
\begin{enumerate}[label=(\roman*)]
\item If $\beta_\Gamma(n)\le\eps n^{d+1}\beta_\Gamma(1)$ for some $n\ge 1$ then $\beta_\Gamma(m)\ll_d(m/n)^{\frac12d(d-1)+1}\beta_\Gamma(n)$ for every $m\ge n$.\label{item:Benj.rel}
\item If $\beta_\Gamma(n)\le\eps n^{d+1}$ for some $n\ge 1$ then $\beta_\Gamma(m)\ll_d(m/n)^d\beta_\Gamma(n)$ for every $m\ge n$.\label{item:growth.abs}
\end{enumerate} 
\end{theorem}
In fact, we need a substantial refinement of \cref{thm:Growth} in order to obtain the isoperimetric inequality we need for \cref{thm:main.unimod.rw}; we obtain this in \cref{prop:growth.lb.rel.all}.

\subsection*{Locality of the escape probability}
Varopoulos's theorem shows that transience is a \emph{global} property of a vertex-transitive graph, determined entirely by the large-scale geometry of that graph. However, using an isoperimetric inequality coming from \cref{thm:Growth}, we can show that the actual \emph{value} of the escape probability is then determined entirely by the \emph{local} geometry of the graph.

To make this precise, define the \emph{local topology} on the space $\mathcal{G}$ of vertex-transitive graphs so that a sequence $(\Gamma_n)_{n=1}^\infty$ of such graphs converges to $\Gamma$ if and only if for each $r\in\N$ there exists $N\in\N$ such that for all $n\ge N$ the balls of radius $r$ in $\Gamma_n$ and $\Gamma$ are isomorphic as rooted graphs. Define $e:\mathcal{G}\to[0,1)$ by setting $e(\Gamma)$ to be the escape probability of the random walk starting at some arbitrary vertex of $\Gamma$. We then have the following result, which answers a question of Itai Benjamini (private communication); we thank Tom Hutchcroft for pointing out that it would follow from a result along the lines of \cref{thm:Growth}.
\begin{corollary}[locality of the escape probability]\label{cor:local}
If $(\Gamma_n)_{n=1}^\infty$ is a sequence of transient vertex-transitive graphs converging locally to a vertex-transitive graph $\Gamma$ then $e(\Gamma_n)\to e(\Gamma)$.
\end{corollary}
The condition that the $\Gamma_n$ are transient is necessary to rule out examples such as $\Gamma_n=\Z^2\times\Z/n\Z$ for $n\in\N$, which satisfies $\Gamma_n\to\Z^3$ but $e(\Gamma_n)=0\not\to e(\Z^3)>0$.

This result is reminiscent of Schramm's famous \emph{locality conjecture} for the critical percolation probability $p_c$, recently proved in breakthrough work of Easo and Hutchcroft \cite{Eas-Hut}. Like transience, for vertex-transitive graphs the property of having $p_c<1$ depends only on the global geometry (it is equivalent to superlinear growth \cite{dgrsy}), but Schramm's conjecture, and now Easo and Hutchcroft's result, asserts that if $(\Gamma_n)_{n=1}^\infty$ is a sequence of vertex-transitive graphs converging locally to a vertex-transitive graph $\Gamma$ such that $p_c(\Gamma_n)<1$ for all $n$ then $p_c(\Gamma_n)\to p_c(\Gamma)$.

\subsection*{Relation to our other work} The present paper contains some of the main applications of a larger project in which we investigate various structural, geometric and probabilistic properties of vertex-transitive graphs. The input we require from elsewhere in that project is Theorem \ref{thm:Growth} above, which results from a finitary and quantitative version of Gromov's famous polynomial growth theorem  \cite[Theorem 1.11]{ttBalls}. A theorem of this form was first proved in a celebrated work of Breuillard, Green and Tao on approximate groups \cite{bgt}, but we needed to significantly refine both the quantitative and the structural conclusions of that theorem for the present application. These refinements turn out to have a number of other applications as well; see \cite[\S 1.4]{ttBalls}.

\subsection*{Further applications} Since we first circulated a preprint of the present paper, a number of applications of its results have emerged. These include universality theorems for cover-time fluctuations \cite{berestycki2023covertime}, a comparison between the mixing time of the interchange process and random walks on transitive graphs \cite{hermon2021interchange}, and non-triviality of the supercritical phase for percolation on finite transitive graphs \cite{hutch-toint}.

\subsection*{Organisation of the paper}   In \S\ref{sec:resistanceStatements}, we restate our results on escape probability in terms of \emph{electric resistance}, which has long been known as a convenient language for studying random walks. Electric resistance is actually a special case of a more general notion called {$p$-resistance}, and a number of our arguments extend without difficulty to this general setting, so we formulate many of our results in this more general setting. 
In \S\ref{sec:resist} we collect some basic facts from potential theory on graphs in preparation for the proofs of our results on $p$-resistance. In \S\ref{sec:bk} we modify and extend Benjamini and Kozma's argument for bounding resistances using isoperimetric inequalities, and in \S\ref{sec:growth->iso} we generalise to vertex-transitive graphs the Coulhon--Saloff-Coste result linking isoperimetric inequalities in Cayley graphs to volume growth. In \S\ref{sec:growth} we deduce lower bounds on the growth of certain vertex-transitive graphs from \cref{thm:Growth}, and in \S\ref{sec:iso} we combine them with the material of \S\ref{sec:growth->iso} to prove Theorems \ref{thm:bk.iso.orig} and \ref{thm:isoperim.orig.rel.bgt} and related results. In \S\ref{sec:resist.ub} we prove the resistance upper bounds of our main theorems, in \S\ref{sec:nash-w} we prove the lower bounds, and in \S\ref{sec:bk.examples} we provide examples to illustrate the optimality of our bounds. Finally, in \S\ref{sec:locality} we prove the locality result \cref{cor:local}.

\subsection*{Miscellaneous notation} We write $\N=\{1,2,\ldots\}$ for the set of positive integers, and write $\N_0=\N\cup\{0\}$.

\subsection*{Acknowledgements} We express our warm and sincere thanks to Itai Benjamini for drawing our attention to the questions that have inspired this project, and for a number of helpful discussions. We also thank Goulnara Arjantseva, Persi Diaconis, Jonathan Hermon and Tom Hutchcroft for helpful conversations, and Russell Lyons and Alain Valette for comments on draft versions of this paper.

\section{Statements (and generalisations) of results in terms of resistance}\label{sec:resistanceStatements}

\subsection{Electric resistance} Our proofs of our bounds on escape probabilities make use of a well-known characterisation of recurrence in terms of \emph{electric resistance}. Indeed, it was in the language of electric resistance that Benjamini and Kozma stated the conjecture that we resolve in \cref{thm:main.rw}. For full details on the links between electrical networks and random walks on graphs we refer the reader to \cite[Chapter 2]{ly-per}; here we give only a very brief overview.

One can identify a connected locally finite graph $\Gamma$ with an electrical network by viewing each edge as a wire with resistance $1$. If $X,Y\subset\Gamma$ are disjoint subsets of $\Gamma$ such that every connected component of $\Gamma\setminus(X\cup Y)$ that borders both $X$ and $Y$ is finite, one can then define the \emph{effective resistance} $R_2(X\leftrightarrow Y)$ between $X$ and $Y$. We define this precisely in Section \ref{sec:resist}; this definition will also explain the subscript $2$ in the notation. See \cite[\S2.2]{ly-per} for an equivalent definition. As before, if either $X$ is a singleton $x$ or $Y$ is a singleton $y$ then we often write instead $R_2(x\leftrightarrow y)$, for example. It is shown in \cite[(2.4)]{ly-per} that
\[
\Prob[\,x\to Y\,]=\frac{1}{\deg(x)R_2(x\leftrightarrow Y)}.
\]

For every vertex $x\in\Gamma$ the sequence $R_2(x\leftrightarrow\Gamma\setminus B_\Gamma(x,r))$ is increasing, and therefore converges to a (possibly infinite) limit $R_2(x\leftrightarrow\infty)$. In particular,
\[
\Prob[\,x\to\infty\,]=\frac{1}{\deg(x)R_2(x\leftrightarrow\infty)}.
\]
Varopoulos's result can therefore be restated as saying that if $\Gamma$ is a connected locally finite vertex-transitive graph with super-quadratic volume growth and $x\in\Gamma$ then $R_2(x\leftrightarrow\infty)<\infty$.

In this language, Theorem \ref{thm:main.unimod.rw} can be restated as follows.
\begin{theorem}\label{thm:main.unimod}
Let $\Gamma$ be a connected, locally finite vertex-transitive graph, let $x\in\Gamma$, and let $r\in\N$ be such that $r<\diam(\Gamma)$. Then
\[
\frac{1}{\deg(\Gamma)}+\frac{r^2}{\deg(\Gamma)\beta_\Gamma(r)}\ll R_2(x\leftrightarrow\Gamma\setminus B_\Gamma(x,r))\ll\frac{1}{\deg(\Gamma)}+\frac{r^2\log r}{\beta_\Gamma(r)}.
\]
\end{theorem}
Theorem \ref{thm:var.converse.rw} is immediate from the following slightly more precise result.
\begin{theorem}\label{thm:var.converse}
There exists $\eps>0$ such that if $\Gamma$ is a locally finite vertex-transitive graph satisfying
\begin{equation}\label{var.conv.hyp}
\beta_\Gamma(n)\le\eps n^3\beta_\Gamma(1)
\end{equation}
for some $n\in\N$ then for every $x\in\Gamma$ we have
\[
R_2(S_\Gamma(x,n)\leftrightarrow S_\Gamma(x,r))\gg\frac{n^2}{\deg(\Gamma)\beta_\Gamma(n)}\log\frac{r}{n}
\]
for every $r\ge n$.
\end{theorem}
\begin{remark*}
One may check that
\[
R_2(x\leftrightarrow S_\Gamma(x,r))\ge R_2(x\leftrightarrow S_\Gamma(x,n))+R_2(S_\Gamma(x,n)\leftrightarrow S_\Gamma(x,r)),
\]
so the conclusion of Theorem \ref{thm:main.unimod} implies in particular that
\[
R(x\leftrightarrow S_\Gamma(x,r))\gg R_2(x\leftrightarrow S_\Gamma(x,n))+\frac{n^2}{\beta_\Gamma(n)}(\log r-\log n).
\]
\end{remark*}

Denoting by $R_{\Gamma,2}$ the maximum resistance between two points in a finite graph $\Gamma$, Theorem \ref{thm:main.rw} can be restated as follows.
\begin{theorem}\label{thm:main}
Let $\Gamma$ be a finite, connected, vertex-transitive graph. Then
\[
\frac{1}{\deg(\Gamma)}+\frac{\diam(\Gamma)^2}{\deg(\Gamma)|\Gamma|}\ll R_{\Gamma,2}\ll\frac{1}{\deg(\Gamma)}+\frac{\diam(\Gamma)^2\log(|\Gamma|/\deg(\Gamma))}{|\Gamma|}.
\]
\end{theorem}

\begin{remark}
Given a finite connected graph $\Gamma$, set $t_\hit(\Gamma)=\max_{x,y\in\Gamma}\E_xT_y$, the maximum expected hitting time between a pair of vertices of $\Gamma$. It is shown in \cite[Proposition 3.18]{aldous-fill} that $t_\hit(\Gamma)\ge|\Gamma|-1$, with equality attained by a complete graph, so $t_\hit(\Gamma)$ is always at least linear in the number of vertices of $\Gamma$. Aldous and Fill \cite[Chapter 4]{aldous-fill} suggest that having a linear \emph{upper} bound on $t_\hit(\Gamma)$ is a natural analogue of transience for finite graphs. In a finite regular graph, a linear upper bound on $t_\hit(\Gamma)$ is equivalent to an upper bound on $R_{\Gamma,2}\deg(\Gamma)$ (this follows from the commute-time identity \cite[Corollary 2.21]{ly-per} and the fact that $t_\hit(\Gamma)\le\max_{x,y\in\Gamma}\E_xT_y+\E_yT_x\le2t_\hit(\Gamma)$). Theorem \ref{thm:main} can therefore also be viewed as an analogue of Varopoulos's theorem in the context of Aldous and Fill's version of `transience'.
\end{remark}

Theorems \ref{thm:main.unimod} and \ref{thm:main} are special cases of Theorems \ref{thm:main.unimodp} and \ref{thm:mainp}, below. We prove Theorem \ref{thm:var.converse} in Section \ref{sec:nash-w}.

Theorems \ref{thm:unimod.linear.rw} and \ref{thm:main.linear.rw} can be similarly rephrased in the language of electric resistance. However, these rephrased statements remain true for a slightly generalised notion of resistance, so we defer them until Theorems \ref{thm:main.unimodp.linear} and \ref{thm:main.linearp}, below, at which point we will have enough terminology to state them in full generality.

\subsection{Generalisation to $p$-parabolicity}\label{subsec:p-parabolic}
The resistance $R_2$ is in fact a special case of a more general notion of \emph{$p$-resistance} $R_p$ for $p\ge1$.
To define this, we first define the \emph{$p$-capacity} $\ca_p(U,U')$ between two subsets $U$ and $U'$ of a graph $\Gamma$ to be the infimum of
\[
\frac{1}{2}\sum_{\{(x,y)\in\Gamma^2:x\sim y\}}|f(x)-f(y)|^p
\]
over all functions $f:\Gamma\to\R$ such that $f= 1$ on $U$ and $f= 0$ on $U'$. By convention we denote by $\ca_p(U,\infty)$ the  infimum of
\[
\frac{1}{2}\sum_{\{(x,y)\in\Gamma^2:x\sim y\}}|f(x)-f(y)|^p
\]
over all finitely supported functions $f:\Gamma\to\R$ such that $f= 1$ on $U$.
The \emph{$p$-resistance} $R_p(U\leftrightarrow U')$ between $U$ and $U'$ is then the reciprocal of the $p$-capacity; thus 
\[
R_p(U\leftrightarrow U') = \frac{1}{\ca_p(U,U')}.
\]
See \S\ref{sec:resist} for a more detailed introduction to $p$-resistance, including all background necessary to understand the present paper.

A connected graph $\Gamma$ is called \emph{$p$-parabolic} if for every vertex $u\in\Gamma$ we have $\ca_p(u,\infty)=0$ (or equivalently if $R_p(u\leftrightarrow \infty)=\infty$ for every vertex $u\in\Gamma$). Otherwise, the graph is called \emph{$p$-hyperbolic}. Thus $2$-parabolicity is equivalent to recurrence. On the other hand, $1$-parabolicity is equivalent to $\Gamma$ being finite.

The following result is a natural generalisation of Theorem \ref{thm:var}, a proof of which for Cayley graphs can be found in \cite[Theorem 2.1]{Maillot}.
\begin{theorem}
Let $p\in(1,\infty)$. Then a locally finite vertex-transitive graph is $p$-parabolic if and only if it has polynomial growth of degree at most $p$. 
\end{theorem}

The following theorem provides a finitary version of this result, and generalises Theorem \ref{thm:main.unimod}. In it and subsequent results we define
\[
\alpha(p)=\begin{cases}
   1&\text{if $p<3$}\\
   1-\frac{p}{\fp+1}&\text{otherwise}.
\end{cases}
\]
\begin{theorem}\label{thm:main.unimodp}
Let $\Gamma$ be a connected, locally finite vertex-transitive graph, let $x\in\Gamma$, and let $r\in\N$ be such that $r<\diam(\Gamma)$. Then for every $p\in(1,\infty)$ we have
\[
\frac{1}{\deg(\Gamma)}+\frac{r^p}{\deg(\Gamma)\beta_\Gamma(r)}\ll R_p(x\leftrightarrow\Gamma\setminus B_\Gamma(x,r))\ll_p\frac{1}{\deg(\Gamma)^{\alpha(p)}}+\frac{r^p(\log r)^{p-1}}{\beta_\Gamma(r)}.
\]
Moreover, if $p$ is not an integer then the upper bound can be improved to
\[
R_p(x\leftrightarrow\Gamma\setminus B_\Gamma(x,r))\ll_p\frac{1}{\deg(\Gamma)^{\alpha(p)}}+\frac{r^p}{\beta_\Gamma(r)}.
\]
\end{theorem} 
We prove the upper bounds of Theorem \ref{thm:main.unimodp} in Section \ref{sec:resist.ub}. We prove the lower bound in Section \ref{sec:nash-w}.

Denoting by $R_{\Gamma,p}$ the maximum $p$-resistance between two vertices of a finite graph $\Gamma$, we have the following variant of Theorem \ref{thm:main.unimodp} for finite graphs, which generalises Theorem \ref{thm:main}.
\begin{theorem}\label{thm:mainp}
Let $\Gamma$ be a finite, connected, vertex-transitive graph. Then for every $p\in(1,\infty)$ we have
\[
\frac{1}{\deg(\Gamma)}+\frac{\diam(\Gamma)^p}{\deg(\Gamma)|\Gamma|}\ll R_{\Gamma,p}\ll_p\frac{1}{\deg(\Gamma)^{\alpha(p)}}+\frac{\diam(\Gamma)^p(\log(|\Gamma|))^{p-1}}{|\Gamma|}.
\]
Moreover, if $p$ is not an integer then the upper bound can be improved to
\[
R_{\Gamma,p}\ll_p\frac{1}{\deg(\Gamma)^{\alpha(p)}}+\frac{\diam(\Gamma)^p}{|\Gamma|}.
\]
\end{theorem}
We prove the upper bounds of Theorem \ref{thm:mainp} in Section \ref{sec:resist.ub}. We state the lower bound of Theorem \ref{thm:mainp} more precisely in Proposition \ref{prop:lower.bound}, which does not in fact require the transitivity hypothesis.

We complement Theorems \ref{thm:main.unimodp} and \ref{thm:mainp} with the following refinement for graphs with sufficiently slow growth.
\begin{theorem}\label{thm:main.unimodp.linear}
Let $p\in(1,\infty)$ and $\eps>0$. Let $\Gamma$ be a connected, locally finite  vertex-transitive graph, let $x\in\Gamma$, and let $r\in\N$ be such that $r<\diam(\Gamma)$. Suppose that $\beta_\Gamma(r)\leq r^{p-\eps}$. Then
\[
R_p(x\leftrightarrow\Gamma\setminus B_\Gamma(x,r))\ll_{p,\eps} \frac{r^p}{\beta_\Gamma(r)}.
\]
\end{theorem}
\begin{theorem}\label{thm:main.linearp}
Let $p\in(1,\infty)$ and $\eps>0$. Let $\Gamma$ be a finite, connected, vertex-transitive graph. Suppose that $\diam(\Gamma)\geq |\Gamma|^{\frac{1}{p-\eps}}$. Then
\[
R_{\Gamma,p}\ll_{p,\eps} \frac{\diam(\Gamma)^p}{|\Gamma|}.
\]
\end{theorem}

\subsection{Benjamini and Kozma's resistance conjecture}Taken together, the upper bounds of Theorems \ref{thm:main} and \ref{thm:main.linearp} completely resolve \cite[Conjecture 4.2]{bk} of Benjamini and Kozma. Indeed, the first part of the conjecture is the upper bound of Theorem \ref{thm:main} with a constant term in place of the $1/\deg(\Gamma)$ term, and so for graphs of large degree Theorem \ref{thm:main} is stronger than that part of the conjecture. The second part of the conjecture is the following.
\begin{corollary}\label{cor:bk.lin}
Let $(\Gamma_n)_{n=1}^\infty$ be a sequence of finite, connected, vertex-transitive graphs with $|\Gamma_n|\to\infty$, and suppose that $\diam(\Gamma_n)=o(|\Gamma_n|)$ as $n\to\infty$. Then $R_{\Gamma_n,2}=o(\diam(\Gamma_n))$ as $n\to\infty$.
\end{corollary}
\begin{proof}
First note that $\deg(\Gamma_n)^{\diam(\Gamma_n)}\ge|\Gamma_n|\to\infty$, so $\max\{\deg(\Gamma_n),\diam(\Gamma_n)\}\to\infty$. Since each of $\deg(\Gamma_n)$ and $\diam(\Gamma_n)$ is bounded below by $1$, this implies that
\begin{equation}\label{eq:deg.diam.to.infty}
\deg(\Gamma_n)\diam(\Gamma_n)\to\infty.
\end{equation}
If $\diam(\Gamma_n)\le|\Gamma|^{3/4}$ then Theorem \ref{thm:main} implies that
\[
\frac{R_{\Gamma_n,2}}{\diam(\Gamma_n)}\ll\frac{1}{\deg(\Gamma)\diam(\Gamma_n)}+\frac{\log(|\Gamma_n|)}{|\Gamma_n|^{1/2}},
\]
which converges to zero by \eqref{eq:deg.diam.to.infty}. On the other hand, if $\diam(\Gamma_n)\ge|\Gamma|^{3/4}$ then $R_{\Gamma_n,2}=o(\diam(\Gamma_n))$ by Theorem \ref{thm:main.linearp}.
\end{proof}

\section{Background on $p$-resistance}\label{sec:resist}
This section is a self-contained introduction to $p$-resistance. Some of this material can be found in Soardi \cite{soardi}, for example.

We start by presenting some basic definitions. Fix $p\in(1,\infty)$. Let $\Gamma$ be an oriented locally finite graph. Given a function $f:\Gamma\to\R$,
we define the \emph{$p$-energy} $\EE_p(f)$ of $f$ via
\[
\EE_p(f)=\frac{1}{2}\sum_{\{(x,y)\in\Gamma^2:x\sim y\}} |f(x)-f(y)|^p.
\]
The $p$-capacity between two disjoint subsets $U$ and $U'$ is thus the infimum of $\EE_p(f)$ over functions $f:\Gamma\to\R$ that equal $1$ on $U$ and $0$ on $U'$. Recall that the $p$-resistance is the inverse of the $p$-capacity. Note that 
if $U_0\subset U_1$ and $U_0'\subset U_1'$, we have \[R_p(U_1\leftrightarrow U'_1)\leq R_p(U_0\leftrightarrow U_0').\] 

We shall give two characterisations of the $p$-resistance that will be useful for our purposes.
 First, we present some basic notions. The \emph{$p$-gradient} of $f:\Gamma\to\R$ is a real-valued function $\nabla_p f$ defined on the oriented edges of the graph by 
\[
\nabla_p f(\bar{e}) =|f(y)-f(x)|^{p-2}(f(y)-f(x))
\]
for an oriented edge $\bar{e}=(x,y)$. When $p=2$ we write $\nabla f$ instead of $\nabla_2f$. Denote by $\langle\cdot,\cdot \rangle_V$ the inner product on $\ell^2(\Gamma)$, and by $\langle\cdot,\cdot  \rangle_E$ the inner product on $\ell^2(E(\Gamma))$.
Note that 
\begin{equation}\label{eq:nablap}
\langle \nabla_p f,\nabla f\rangle_E= \EE_p(f).
\end{equation}
The \emph{divergence} of a function $F$ defined on the set of oriented edges is a function $\di\; F:\Gamma\to\R$ defined via $\di\; F(x) = \sum_{y\sim x} F(x,y).$
We define the \emph{$p$-Laplacian} $\Delta_p f$ of $f$ by 
\[\Delta_p f =-\di (\nabla_p f).\]
Equivalently, for every $x\in\Gamma$ we have
\[\Delta_p f(x) = \sum_{\{(x,y)\in\Gamma^2:x\sim y\}} |f(x)-f(y)|^{p-2}(f(x)-f(y)).\]
We say that $f$ is \emph{$p$-harmonic} at $x$ if $\Delta_p f(x)=0$. Using the fact that $-\di$ and $\nabla$ are adjoint to one another, we have 
\begin{equation}\label{eq:adj}
\langle \nabla_p f,\nabla f\rangle_{E}=\langle \Delta_p f,f\rangle_{V}=\sum_{u\in U}\Delta_p f(u)
\end{equation}

The divergence satisfies the following well-known discrete version of Stokes's theorem.
\begin{prop}[Stokes's theorem]\label{prop:stokes}
Let $\Gamma$ be a locally finite connected graph and let $A\subset\Gamma$ be a finite subset. For every $e\in \partial^E A$, let $\bar{e}$ be the oriented edge starting inside $A$ and ending outside $A$. Then
\[
\sum_{a\in A} \di\; F(a) = \sum_{e\in \partial^E A} F(\bar{e}).
\]
\end{prop}
\begin{proof}
We have
\[
\sum_{a\in A} \di\; F(a) = \langle \di\; F, 1_{A}\rangle_V,
\]
and since $\di$ is the adjoint of $-\nabla$ we have
\[
\langle \di\; F, 1_{A}\rangle_V= -\langle F, \nabla 1_{A}\rangle_E=\sum_{e\in \partial^E A} F(\bar{e}).
\]
\end{proof}
We deduce from Proposition \ref{prop:stokes} that for every $f:\Gamma\to\R$ we have
\begin{equation}\label{eq:stokes}
\sum_{a\in A}  \Delta_p f(a) =  \sum_{e\in \partial^E A} \nabla_p f(\bar{e})
\end{equation}

The following is a classical fact. 
\begin{prop}\label{prop:current}
Given two disjoint subsets $U$ and $U'$ of a locally finite graph $\Gamma$ such that every connected component of $\Gamma\setminus(U\cup U')$ that borders both $U$ and $U'$ is finite. For every $t>0$, there exists a unique minimizer of $\EE_p(f)$ over functions $f$ that equal $t$ on $U$ and $0$ on $U'$. Moreover, this function is $p$-harmonic on $\Gamma\setminus(U\cup U')$. \end{prop}
\begin{proof}
Since we can treat every connected components of $\Gamma\setminus(U\cup U')$ independently, we can assume without loss of generality that $\Gamma\setminus(U\cup U')$ is finite.
Note that $f\to \EE_p(f)$ is a non-negative, strictly convex and smooth function on the finite dimensional affine space of functions  that equal $t$ on $U$ and $0$ on $U'$. Therefore it reaches a unique minimum $f$ that is in particular a critical value of $\EE_p$ on this affine space. A direct calculation shows that the differential of $\EE_p$ at $f$ satisfies
\[d\EE_p(f) (h)=p\langle\Delta_p f,h \rangle_V,\] 
for every function $h$ that equals $0$ on $U\cup U'$. Therefore $f$ is $p$-harmonic on $\Gamma\setminus(U\cup U')$ as claimed.
\end{proof}

A function as in \cref{prop:current} is called a \emph{$p$-potential} $f$ from a subset $U$ to a subset $U'$. If $t=1$ then $f$ is said to be a \emph{unit} $p$-potential. Note that if $f$ is a unit $p$-potential from $U$ to $U'$, then $(1-f)$ is a unit potential from $U'$ to $U$.
Note that if $f$ is a unit $p$-potential from $U$ to $U'$, then $(1-f)$ is a unit potential from $U'$ to $U$. The \emph{$p$-current} of a potential through an oriented edge $\bar{e}$ is then defined to be $\nabla_p f(\bar{e})$, and the \emph{total $p$-current} $C_p(f)$ of $f$ is defined via
\[
C_p(f)=\sum_{e\in \partial^E U} \nabla_p f(\bar{e}),
\]
where edges are oriented outwards.
Note that by (\ref{eq:stokes}) we have $C_p(f)=\sum_{e\in \partial^E A} \nabla_p f(\bar{e})$ for every subset $A$ of vertices containing $U$ and disjoint from $U'$.
In the particular case where $U$ and $U'$ are single vertices $u$ and $v$, the total $p$-current $C_p(f)$ of a $p$-potential $f$ is simply 
$\Delta_p f(u)=-\Delta_p f(v).$

\begin{prop}\label{prop:currentResistance}
Given two disjoint subsets $U$ and $U'$ of a locally finite graph $\Gamma$ such that every connected component of $\Gamma\setminus(U\cup U')$ that borders both $U$ and $U'$ is finite, we have
\[R_p(U\leftrightarrow U')= \frac{1}{C_p(f)},\]
where $f$ is the unit $p$-potential from $U$ to $U'$.
In other words, the $p$-capacity coincides with the total $p$-current. 
\end{prop}

\begin{proof}
This is classical, but for the sake of completeness we sketch its proof.  First, one checks that since every connected component of $\Gamma\setminus(U\cup U')$ that borders both $U$ and $U'$ is finite, there exists a unique function $f$ that is $p$-harmonic outside $U\cup U'$ and equal to $1$ on $U$ and $0$ on $U'$. Moreover, $f$ realizes the infimum in the definition of $p$-capacity.
We deduce from \eqref{eq:adj} and \eqref{eq:stokes} that
\[
\sum_{u\in U}\Delta_p f(u)=\sum_{u\in U}\Delta_p f(u)=\sum_{e\in \partial^E U} \nabla_p f(\bar{e})=C_p(f). 
\]
So the conclusion follows from \eqref{eq:nablap}.
\end{proof}

\begin{prop}\label{prop:p-energy/p-resistance}
Let  $U$ and $U'$ be two disjoint subsets of a locally finite graph $\Gamma$ such that every connected component of $\Gamma\setminus(U\cup U')$ that borders both $U$ and $U'$ is finite, and let $f$ be a $p$-potential between $U$ and $U'$ taking the value $t>0$ on $U$. Suppose that $C_p(f)=1$. Then 
\[R_p(U\leftrightarrow U')=t^{p-1}=\EE_p(f)^{p-1}.\]
\end{prop}
\begin{proof}
Note that $\nabla_p$ is homogeneous of degree $p-1$. Hence, given a $p$-potential $f$ from $U$ to $U'$, \cref{prop:currentResistance} implies that $R_p(U\leftrightarrow U')=\frac{t^{p-1}}{C_p(f)}$. The fact that $C_p(f)=1$ therefore gives the first equality.
Using \eqref{eq:nablap} and \eqref{eq:adj} obtain
\[\EE_p(f)=\langle \nabla_p f,\nabla f\rangle_{E}=\langle \Delta_p f,f\rangle_{V}=\sum_{u\in U}(\Delta_p f(u))f(u) .\]
Since $f$ is a $p$-potential from $U$ to $U'$, and using that $\sum_{u\in U}\Delta_p f(u)=C_p(f)$, this gives
\[\EE_p(f)=tC_p(f).\]
Since $C_p(f)=1$, we conclude that $\EE_p(f)=t$, giving the second equality.
\end{proof}

\section{An isoperimetric bound on $p$-resistance}\label{sec:bk}
In this section we use an argument of Benjamini and Kozma \cite{bk} to bound certain $p$-resistances in terms of isoperimetric inequalities. Given a set $A$ in a graph $\Gamma$, we write $\partial^EA$ for the \emph{edge boundary} of $A$, defined to consist of all those edges having one endpoint in $A$ and one endpoint outside of $A$. Note that
\begin{equation}\label{eq:boundaries}
|\partial A|\le|\partial^EA|\le\deg(\Gamma)|\partial A|
\end{equation}
for an arbitrary finite subset $A$ of an arbitrary graph $\Gamma$. Throughout this section, given a subset $A$ of a graph $\Gamma$ we write
\[
j_{A,p}=\min\left\{\frac{|A|}{|\partial A|^{\frac{p}{p-1}}}+\frac{1}{|\partial A|^{\frac{1}{p-1}}},\frac{\deg(\Gamma)|A|}{|\partial^EA|^{\frac{p}{p-1}}}+\frac{1}{|\partial^EA|^{\frac{1}{p-1}}}\right\}.
\]

The first specific result we give is the following, which bounds the $p$-resistance between a pair of points in a finite graph.
\begin{theorem}[Benjamini--Kozma {\cite[Theorem 2.1]{bk}}]\label{thm:bk}
Let $\Gamma$ be a finite connected graph, and let $u,v\in\Gamma$ be such that each of $\Gamma\setminus\{u\}$ and $\Gamma\setminus\{v\}$ is connected. Then for every $p\in(1,\infty)$ we have
\[
\begin{split}
R_p(u\leftrightarrow v)^{\frac{1}{p-1}}\ll\frac{1}{\deg(u)^{\frac{1}{p-1}}}+\frac{1}{\deg(v)^{\frac{1}{p-1}}}+\sum_{n=1}^{\lfloor\log_2(|\Gamma|/\deg(u))\rfloor}\max_{\substack{A\subset\Gamma\\u\in A\\A\text{ connected}\\|\Gamma|/2^{n+1}<|A|\le|\Gamma|/2^n}}j_{A,p}\\
+\sum_{n=1}^{\lfloor\log_2(|\Gamma|/\deg(v))\rfloor}\max_{\substack{A\subset\Gamma\\v\in A\\A\text{ connected}\\|\Gamma|/2^{n+1}<|A|\le|\Gamma|/2^n}}j_{A,p}.
\end{split}
\]
\end{theorem}
With some specific applications in mind we have stated Theorem \ref{thm:bk} in a slightly stronger form than Benjamini and Kozma, but our proof is essentially identical to theirs. The same argument also gives the following bound in a possibly infinite graph.
\begin{theorem}\label{thm:bk.inf}
Let $\Gamma$ be a connected, locally finite graph. Let $u\in\Gamma$, and let $B$ be a finite connected proper subset of $\Gamma$ containing $u$. Then for every $p>1$ we have
\[
R_p(u\leftrightarrow\partial B)^{\frac{1}{p-1}}\ll\frac{1}{\deg(u)^{\frac{1}{p-1}}}+\sum_{n=0}^{\lfloor\log_2(|B|/\deg(u))\rfloor}\max_{\substack{A\subset B\\u\in A\\A\text{ connected}\\|B|/2^{n+1}<|A|\le|B|/2^n}}j_{A,p}.
\]
\end{theorem}

\begin{remark*}
Which term of $j_{A,p}$ achieves the minimum depends on the context: the strongest relations between the two terms that one can obtain from \eqref{eq:boundaries} are
\[
\begin{split}
\frac{1}{\deg(\Gamma)}\left(\frac{\deg(\Gamma)|A|}{|\partial^EA|^{\frac{p}{p-1}}}+\frac{1}{|\partial^EA|^{\frac{1}{p-1}}}\right)\le\left(\frac{|A|}{|\partial A|^{\frac{p}{p-1}}}+\frac{1}{|\partial A|^{\frac{1}{p-1}}}\right)
\qquad\qquad\qquad\qquad\qquad\\
\le\deg(\Gamma)^{\frac{1}{p-1}}\left(\frac{\deg(\Gamma)|A|}{|\partial^EA|^{\frac{p}{p-1}}}+\frac{1}{|\partial^EA|^{\frac{1}{p-1}}}\right),
\end{split}
\]
both of which are sharp. Indeed, see Remark \ref{rem:vertex} for a specific example of a graph in which the vertex term is stronger by this amount, and Remark \ref{rem:edge} for an example in which the edge term is stronger by this amount.
\end{remark*}

We will deduce Theorems \ref{thm:bk} and \ref{thm:bk.inf} from the following result.
\begin{prop}\label{prop:bk}
Let $\Gamma$ be a finite connected graph. Let $u\in\Gamma$, and let $U'\subset\Gamma$ be a non-empty subset such that $\Gamma\setminus U'$ is connected and contains $u$. Let $\lambda>0$, and suppose that $|\Gamma\setminus U'|\ge\lfloor\lambda\rfloor$. Let $f$ be the unit $p$-potential from $U'$ to $u$. Let $u'\in\partial(\Gamma\setminus U')$, and label the elements of $(\Gamma\setminus U')\cup\{u'\}$ as $x_1,\ldots,x_k$ in such a way that $f(x_m)$ is non-decreasing in $m$. Then
\[
f(x_{\lfloor\lambda\rfloor+1})\ll C_p(f)^{\frac{1}{p-1}}\left(\frac{1}{\deg(u)^{\frac{1}{p-1}}}+\sum_{n=0}^{\lfloor\log_2(\lambda/\deg(u))\rfloor}\max_{\substack{A\subset\Gamma\setminus U'\\u\in A\\A\text{ connected}\\\lambda/2^{n+1}<|A|\le\lambda/2^n}}j_{A,p}\right).
\]
\end{prop}

\begin{proof}We follow Benjamini and Kozma \cite[Theorem 2.1]{bk}.
Write $A_m=\{x_1,\ldots,x_m\}$ and $\theta(m)=f(x_m)$ for $m=1,\ldots,\lfloor\lambda\rfloor$. Note that by the maximum principle we may assume that each set $A_m$ is connected. 

It follows from Proposition \ref{prop:stokes} that the $p$-currents along all of the edges in $\partial^EA_m$ sum to $C_p(f)$, and hence that at least half of those edges carry a $p$-current of at most $2C_p(f)/|\partial^EA_m|$ each. Between them, these edges must meet at least $|\partial^EA_m|/2\deg(\Gamma)$ vertices outside $A_m$, so
\begin{equation}\label{eq:increment.edge}
\theta\left(m+\left\lceil\frac{|\partial^EA_m|}{2\deg(\Gamma)}\right\rceil\right)\le\theta(m)+\left(\frac{2C_p(f)}{|\partial^EA_m|}\right)^{\frac{1}{p-1}}.
\end{equation}
Similarly, if we select, for each $y\in\partial A_m$, an edge joining $y$ to $A_m$, then at least half of these edges carry a current of at most $2C_p(f)/|\partial A_m|$ each, and so
\begin{equation}\label{eq:increment.vertex}
\theta\left(m+\left\lceil\frac{|\partial A_m|}{2}\right\rceil\right)\le\theta(m)+\left(\frac{2C_p(f)}{|\partial A_m|}\right)^{\frac{1}{p-1}}.
\end{equation}
Using the same argument one final time, at least half of the edges containing $u$ carry a current of at most $2C_p(f)/\deg(u)$ each, and so
\begin{equation}\label{eq:increment.base}
\theta\left(1+\left\lceil\frac{\deg(u)}{2}\right\rceil\right)\le\left(\frac{2C_p(f)}{\deg(u)}\right)^{\frac{1}{p-1}}.
\end{equation}
Set
\[
s_n=\min\{|\partial^EA|:A\subset X;u\in A;A\text{ connected};\lfloor\lambda/2^{n+1}\rfloor<|A|\le\lfloor\lambda/2^n\rfloor\}
\]
and
\[
t_n=\min\{|\partial A|:A\subset X;u\in A;A\text{ connected};\lfloor\lambda/2^{n+1}\rfloor<|A|\le\lfloor\lambda/2^n\rfloor\}
\]
for each $n=1,\ldots,\lfloor\log_2\lambda\rfloor$. It then follows from \eqref{eq:increment.edge} that
\[
\theta\left(\left\lfloor\frac{\lambda}{2^n}\right\rfloor+1\right)\le\theta\left(\left\lfloor\frac{\lambda}{2^{n+1}}\right\rfloor+1\right)+\left(\frac{2}{s_n}\right)^{\frac{1}{p-1}}\left(\frac{2\deg(\Gamma)\lceil\lambda/2^{n+1}\rceil}{s_n}+1\right)C_p(f)^{\frac{1}{p-1}},
\]
and from \eqref{eq:increment.vertex} that
\[
\theta\left(\left\lfloor\frac{\lambda}{2^n}\right\rfloor+1\right)\le\theta\left(\left\lfloor\frac{\lambda}{2^{n+1}}\right\rfloor+1\right)+\left(\frac{2}{t_n}\right)^{\frac{1}{p-1}}\left(\frac{2\lceil\lambda/2^{n+1}\rceil}{t_n}+1\right)C_p(f)^{\frac{1}{p-1}},
\]
for each $n=0,\ldots,\lfloor\log_2\lambda\rfloor$. The desired result follows from these inequalities and \eqref{eq:increment.base}.
\end{proof}

\begin{proof}[Proof of Theorem \ref{thm:bk.inf}]
We may restrict attention to the subgraph $\Gamma'$ of $\Gamma$ induced by $B\cup\partial B$, noting that $\Gamma'$ is finite. The theorem then follows from applying Proposition \ref{prop:bk} in $\Gamma'$ with $U'=\partial B$ and $\lambda=|B|$, noting that $f(x_{\lambda+1})=1$.
\end{proof}

\begin{proof}[Proof of Theorem \ref{thm:bk}]
First apply Proposition \ref{prop:bk} with $\lambda=|\Gamma|/2$ to conclude that
\begin{equation}\label{eq:bk.pf.1}
f(x_{\lfloor|\Gamma|/2\rfloor+1})\ll C_p(f)^{\frac{1}{p-1}}\left(\frac{1}{\deg(u)^{\frac{1}{p-1}}}+\sum_{n=1}^{\lfloor\log_2(|\Gamma|/\deg(u))\rfloor}\max_{\substack{A\subset X\\u\in A\\A\text{ connected}\\|\Gamma|/2^{n+1}<|A|\le|\Gamma|/2^n}}j_{A,p}\right).
\end{equation}
Next, note that if we interchange the roles of $u$ and $v$ then this amounts to replacing $f$ with $1-f$, and so applying Proposition \ref{prop:bk} with $\lambda=|\Gamma|/2$ again, together with the fact that $C_p(f)=C_p(1-f)$, implies that
\begin{equation}\label{eq:bk.pf.2}
(1-f(x_{\lfloor|\Gamma|/2\rfloor+1}))\ll C_p(f)^{\frac{1}{p-1}}\left(\frac{1}{\deg(v)^{\frac{1}{p-1}}}+\sum_{n=1}^{\lfloor\log_2(|\Gamma|/\deg(v))\rfloor}\max_{\substack{A\subset X\\v\in A\\A\text{ connected}\\|\Gamma|/2^{n+1}<|A|\le|\Gamma|/2^n}}j_{A,p}\right).
\end{equation}
The theorem then follows from combining \eqref{eq:bk.pf.1} and \eqref{eq:bk.pf.2}.
\end{proof}

\section{Isoperimetric inequalities from lower bounds on growth}\label{sec:growth->iso}
Coulhon and Saloff-Coste \cite{csc} famously showed that lower bounds on the sizes of balls in Cayley graphs lead to isoperimetric inequalities. In this section we show how their result extends to all locally finite vertex-transitive graphs.

Throughout the section we assume familiarity with our earlier paper \cite{tt.trof}. In particular, given a locally finite vertex-transitive graph $\Gamma$, the group $\Aut(\Gamma)$ of automorphisms of $\Gamma$ is a locally compact group with respect to the topology of pointwise convergence, and every closed subgroup $\Aut(\Gamma)$ is also a locally compact group in which vertex stabilisers are compact and open. Moreover, an arbitrary closed subgroup $G<\Aut(\Gamma)$ admits a \emph{Haar measure} $\mu$, the properties of which include that
\begin{enumerate}[label=(\arabic*)]
\item$\mu(V)<\infty$ if $V$ is compact,
\item$\mu(U)>0$ if $U$ is open and nonempty,
\item$\mu(gA)=\mu(A)$ for every Borel set $A\subset G$ and every $g\in G$, and
\item\label{item:Haar.unique} if $\mu'$ is another Haar measure on $G$ then there exists $\lambda>0$ such that $\mu'=\lambda\cdot\mu$.
\end{enumerate}
See \cite[\S15]{hew-ross} for a detailed introduction to Haar measures.

Given a locally compact group $G$ with Haar measure $\mu$, we define the space $L^1(G)$ with respect to $\mu$. Note that since a right translate of a Haar measure is again a Haar measure, by property \ref{item:Haar.unique} there exists a homomorphism $\Delta_G:G\to\R^+$, called the \emph{modular function}, such that
\[
\mu(Ag)=\Delta_G(g^{-1})\mu(A)
\]
for every Borel set $A$. This in turn implies that we may define an action of $G$ on $L^1(G)$ via $gf(x)=f(xg)$. Note that property \ref{item:Haar.unique} implies that neither $L^1(G)$ nor $\Delta_G$ depend on the choice of Haar measure $\mu$. A locally compact group $G$ is called \emph{unimodular} if $\Delta_G\equiv1$, in which case $\mu(gA)=\mu(Ag)=\mu(A)$ for every Borel set $A\subset G$ and every $g\in G$.

The following proposition generalises the Coulhon--Saloff-Coste result to vertex-transitive graphs. The unimodular case was already known \cite[Lemma 10.46]{ly-per}; the non-unimodular case appears to be new.
\begin{prop}\label{prop:iso-growth}
Let $\Gamma$ be a locally finite vertex-transitive graph, and define $\phi_\Gamma:\N\to\N$ by setting $\phi_\Gamma(m)$ equal to the minimum $r\in\N$ such that $\beta_\Gamma(r)\ge m$. Let $\lambda\approx0.31517$ be the unique solution to $\lambda(\lambda+2)e^\lambda=1$. Then for every finite subset $A\subset\Gamma$ with $|A|\le|\Gamma|/2$ we have
\begin{equation}\label{eq:iso-growth}
|\partial A|\ge\frac{\lambda|A|}{\phi_\Gamma(2|A|)}.
\end{equation}
If at least one closed unimodular subgroup of $\Aut(\Gamma)$ acts transitively on $\Gamma$ then $\lambda$ can be replaced by $\frac12$ in \eqref{eq:iso-growth}.
\end{prop}

\begin{remark}\label{rem:vertex}Proposition \ref{prop:iso-growth} trivially implies an edge isoperimetric inequality by \eqref{eq:boundaries}. In general this cannot be improved. For example, if $G=(\Z/n\Z)^d\times(\Z/k\Z)$ and
\[
S=\Big(\{-1,0,1\}^d\times\{0\}\Big)\cup\Big(\{0\}^d\times(\Z/k\Z)\Big)
\]
then Proposition \ref{prop:iso-growth} and \eqref{eq:boundaries} give bounds of the form $|\partial^EA|\gg k^{1/d}|A|^{1-1/d}$, with equality (up to multiplicative constants) achieved by sets of the form $([-r,r]\cap\Z)^d\times(\Z/k\Z)$. In particular, if using using an isoperimetric inequality coming from Proposition \ref{prop:iso-growth} to obtain a bound on resistance from Theorem \ref{thm:bk}, it is always better to use the vertex isoperimetric inequality than the edge isoperimetric inequality.
\end{remark}

We start by proving an analogue of Proposition \ref{prop:iso-growth} for locally compact groups. Given a group $G$ with a symmetric generating set $S$, we write $\partial_SA$ for the external vertex boundary of a set $A\subset G$ in the Cayley graph $\Cay(G,S)$. Given $g\in G$, we also write $|g|_S$ for the distance of $g$ from the identity in $\Cay(G,S)$.

\begin{prop}\label{prop:iso-growth.lc}
Let $G$ be a locally compact group with a left Haar measure $\mu$ and a precompact symmetric open generating set $S$ containing the identity, and define $\phi_S:[0,\infty)\to\N$ by setting $\phi_S(\xi)$ equal to the minimum $r\in\N$ such that $\mu(S^r)\ge\xi$. Let $\lambda\approx0.31517$ be the unique solution to $\lambda(\lambda+2)e^\lambda=1$. Then for every precompact open set $A\subset G$ with $\mu(A)\le\mu(G)/2$ we have
\[
\mu(\partial_SA)\ge\frac{\lambda\mu(A)}{\phi_S(2\mu(A))}.
\]
If $G$ is unimodular then this can be improved to
\[
\mu(\partial_SA)\ge\frac{\mu(A)}{2\phi_S(2\mu(A))}.
\]
\end{prop}
\begin{proof}
Let $r\in \N$, and suppose that $A$ is a measurable subset such that $\mu(A)\le\mu(S^r)/2$. We need to show that
\begin{equation}\label{eq:modularCase}
\mu(\partial_SA)\geq \frac{\lambda\mu(A)}{r},
\end{equation}
or
\begin{equation}\label{eq:unimodularCase}
\mu(\partial_SA)\geq \frac{\mu(A)}{2r}
\end{equation}
if $G$ is unimodular.

Assume first that there exists $s\in S$ such $\Delta_G(s)\geq 1+\frac{\lambda}{r}$. Then for any measurable subset $A\subset G$ with finite measure, we have
\[
\mu(\partial_SA)\ge\mu(As^{-1}\setminus A)\geq \mu(As^{-1})-\mu(A)=(\Delta_G(s)-1)\mu(A)\geq \frac{\lambda\mu(A)}{r},
\]
giving \eqref{eq:modularCase} as required.

Now assume that $\Delta_G(s)\leq 1+\frac{\lambda}{r}$ for every $s\in S$. This implies that for every $g\in S^r$ we have
\begin{equation}\label{eq:mod.log2}
\Delta_G(g)\leq \left(1+\frac{\lambda}{r}\right)^r\leq e^\lambda.
\end{equation}
Define a linear operator $M:L^1(G)\to L^1(G)$ via $M(f)(x)=\E_{g\in S^r}\,f(xg)$. We claim that
\begin{equation}\label{eq:Mf}
\|1_A-M(1_A)\|_1\geq\mu(A).
\end{equation}
Indeed, write $F^+=\max\{1_A-M(1_A),0\}$ and $F^-=\min\{1_A-M(1_A),0\}$, so that $\|1_A-M(1_A)\|_1=\|F^+\|_1+\|F^-\|_1$. Since $\mu(A)\le\mu(S^r)/2$ we have $M(1_A)(x)\leq 1/2$ for all $x\in G$, so $\|F^+\|_1\ge\mu(A)/2$. To see that $\|F^-\|_1\ge\mu(A)/2$ also, write
\[
M(1_A)(x)=\frac1{\mu(S^r)}\int_{a\in A}1_{aS^r}(x)\diff\mu(a)
\]
noting that therefore
\begin{align*}
\|F^-\|_1&\ge\frac1{\mu(S^r)}\int_{x\in G\setminus A}\int_{a\in A}1_{aS^r}(x)\diff\mu(a)\diff\mu(x)\\
	&=\frac1{\mu(S^r)}\int_{x\in G}\int_{a\in A}1_{aS^r\setminus A}(x)\diff\mu(a)\diff\mu(x)\\
	&=\frac1{\mu(S^r)}\int_{a\in A}\int_{x\in G}1_{aS^r\setminus A}(x)\diff\mu(x)\diff\mu(a)&&\text{(by Fubini)}\\
	&\ge\frac1{\mu(S^r)}\int_{a\in A}\textstyle\frac12\mu(S^r)\diff\mu(a)&&\text{(because $\mu(A)\le\textstyle\frac12\mu(S^r)$)}\\
	&=\mu(A)/2.
\end{align*}

On the other hand, for all $g=s_1\ldots s_r\in S^r$, and all $f\in L^{1}(G)$, we deduce by the triangle inequality that
\[
\|f-gf\|_1\leq  \sum_{i=0}^{r-1}\|s_{r-i+1}\ldots s_rf-s_{r-i}\ldots s_rf\|_1= \sum_{i=0}^{r-1}\Delta_G(s_{r-i+1}\ldots s_r)\|f-s_{r-i}f\|_1.
\]
By \eqref{eq:mod.log2}, this implies that
\[\|f-gf\|_1\leq e^{\lambda}r\sup_{s\in S}\|f-sf\|_1,\]
or
\[\|f-gf\|_1\leq r\sup_{s\in S}\|f-sf\|_1\]
if $G$ is unimodular.
By the triangle inequality this in turn yields
\[\|f-M(f)\|_1\leq  e^{\lambda}r\sup_{s\in S}\|f-sf\|_1,\]
or
\[\|f-M(f)\|_1\leq  r\sup_{s\in S}\|f-sf\|_1\]
if $G$ is unimodular.
Applying this to $f=1_A$, and combining it with (\ref{eq:Mf}), we obtain
\[\sup_{s\in S}\mu(A\vartriangle As)\geq \frac{\mu(A)}{e^{\lambda}r},\]
or 
\[\sup_{s\in S}\mu(A\vartriangle As)\geq \frac{\mu(A)}{r}\]
if $G$ is unimodular.
Since
\begin{align*}
\mu(A\vartriangle As)&=\mu(A\setminus As)+\mu(As\setminus A)\\
   &=\Delta_G(s)\mu(As^{-1}\setminus A)+\mu(As\setminus A)\\
   &\leq (2+\textstyle\frac{\lambda}{r})\sup_{s'\in S}\mu(As'\setminus A)\\
   &\leq (2+\lambda)\mu(\partial_SA)
\end{align*}
for all $s\in S$, or $\mu(A\vartriangle As)\le2\mu(\partial_SA)$ for all $s\in S$ if $G$ is unimodular, this gives \eqref{eq:modularCase} or \eqref{eq:unimodularCase} as appropriate, completing the proof.
\end{proof}

We now move on to the proof of Proposition \ref{prop:iso-growth}. For the remainder of this section, $\Gamma$ always denotes a locally finite vertex-transitive graph, $e\in\Gamma$ is some distinguished vertex, $G<\Aut(\Gamma)$ is a closed, vertex-transitive subgroup, and $S=\{g\in G:d(g(e),e)\le1\}$. By \cite[Lemma 4.8]{tt.trof}, the set $S$ is a compact open generating set for $G$ containing the identity. By transitivity of $G$ we may pick, for each $x\in\Gamma$, an automorphism $g_x\in G$ such that $g_x(e)=x$. Write $G_e$ for the stabiliser of $e$ in $G$, and given an arbitrary subset $X\subset\Gamma$, write $G_{e\to X}=\{g\in G:g(e)\in X\}$, noting that
\begin{equation}\label{eq:Ug_x}
G_{e\to X}=\bigcup_{x\in X}g_xG_e.
\end{equation}
The stabiliser $G_e$ is open by definition, and compact by \cite[Lemma 4.4]{tt.trof}, so \eqref{eq:Ug_x} means that $G_{e\to X}$ is compact and open whenever $X$ is finite. Normalising the Haar measure $\mu$ on $G$ so that $\mu(G_e)=1$, \eqref{eq:Ug_x} also means that
\begin{equation}\label{eq:pullback.mu}
\mu(G_{e\to X})=|X|.
\end{equation}
Writing $X^+=\{y:d(y,X)\le1\}$ for the neighbourhood of $X$, for every $g\in G$ we have
\begin{align*}
g\in G_{e\to X}S&\iff\text{there exists $q\in G_{e\to X}$ such that $d(q^{-1}g(e),e)\le1$}\\
    &\iff\text{there exists $q\in G_{e\to X}$ such that $d(g(e),q(e))\le1$}\\
    &\iff g(e)\in X^+\\
    &\iff g\in G_{e\to X^+},
\end{align*}
and so
\begin{equation}\label{eq:mu.boundary}
\begin{array}{rcl}\mu(\partial_SG_{e\to X})&=&\mu(G_{e\to\partial X})\\
   &=&|\partial X|
\end{array}
\end{equation}
by \eqref{eq:pullback.mu}.

\begin{proof}[Proof of Proposition \ref{prop:iso-growth}]
Let $G$ be a closed unimodular subgroup of $\Aut(\Gamma)$ acting transitively on $\Gamma$ if one exists; otherwise let $G=\Aut(\Gamma)$. Normalising the Haar measure so that  $\mu(G_e)=1$, it follows from \cite[Lemma 4.8]{tt.trof} that the function $\phi_S$ defined in \cref{prop:iso-growth.lc} satisfies
\begin{equation}\label{eq:phi.S.Gamma}
\phi_S(m)=\phi_\Gamma(m)
\end{equation}
for every $m\in\N$. Given a finite set $A\subset\Gamma$ with $|A|\le|\Gamma|/2$, we have $G_{e\to A}$ compact and open as noted above and $\mu(G_{e\to A})\le\mu(G)/2$ by \eqref{eq:pullback.mu}, and hence
\begin{align*}
|\partial A|&=\mu(\partial_SG_{e\to A})&\text{(by \eqref{eq:mu.boundary})}\phantom{,}\\
    &\ge\frac{\lambda\mu(G_{e\to A})}{\phi_S(2\mu(G_{e\to A}))}&\text{(by \cref{prop:iso-growth.lc})}\phantom{,}\\
    &=\frac{\lambda|A|}{\phi_\Gamma(2|A|)}&\text{(by \eqref{eq:phi.S.Gamma} and \eqref{eq:pullback.mu})},
\end{align*}
as required for the general case. If $G$ is unimodular then the $\lambda$ coming from \cref{prop:iso-growth.lc} in the second line above, and hence that in the third line, can be replaced by $\frac12$, as required.
\end{proof}

\section{Growth bounds}\label{sec:growth}

In light of Proposition \ref{prop:iso-growth} and the results of Section \ref{sec:bk}, in principle one can bound resistances in a vertex-transitive graph by proving lower bounds on the sizes of balls in that graph. In this section we prove the following results of this type.

\begin{prop}\label{prop:growth.lb}
Let $q\ge1$, let $r\in\N$, and suppose that $\Gamma$ is a locally finite vertex-transitive graph such that $\beta_\Gamma(r)\ge r^q$. Then for every $n\le r^{\{q\}}$ we have $\beta_\Gamma(n)\gg_\fq n^{\lfloor q\rfloor+1}$, whilst for every $n$ with $r^{\{q\}}\le n\le r$ we have $\beta_\Gamma(n)\gg_\fq r^{\{q\}}n^{\lfloor q\rfloor}$.
\end{prop}
Proposition \ref{prop:growth.lb} combines with Proposition \ref{prop:iso-growth} to imply a refined version of an isoperimetric inqeuality conjectured by Benjamini and Kozma, which we state below as Theorem \ref{thm:bk.iso}. It also answers a question that was posed to us by Jonathan Hermon. Proposition \ref{prop:growth.lb} is easily seen to be optimal by considering the Cayley graph of $\Z^\fq\times\Z/\lceil r^{\{q\}}\rceil\Z$ with respect to the generating set $\{-1,0,1\}^{\fq+1}$.

For the next result, and throughout the rest of this section, we define a function
\[
\begin{array}{ccccl}
h&:&\N_0&\to&\N_0\vspace{6pt}\\
   &&     d    &\mapsto&\begin{cases}
                                       0 & \text{if $d=0$}\\
                                       1+\textstyle\frac{1}{2}d(d-1) &\text{otherwise.}
                                       \end{cases}
\end{array}
\]
Thus, $h(d)$ is the maximum homogeneous dimension of a nilpotent Lie group of dimension $d$.

\begin{prop}\label{prop:growth.lb.rel.all}
Let $q\ge1$, let $r\in\N$, and suppose that $\Gamma$ is a locally finite vertex-transitive graph such that $\beta_\Gamma(r)\ge r^q\beta_\Gamma(1)$. Then, setting
\[
b(q)=\begin{cases}q &\text{if $q\in[0,3]\cup\{4\}$},\\
    \max\{d\in\N:h(d-1)\le q\}&\text{otherwise},
    \end{cases}
\]
we have
\begin{equation}\label{eq:lb.rel.bgt}
\beta_\Gamma(n)\gg_\fq n^{b(q)}\beta_\Gamma(1)
\end{equation}
for every $n\in\N$ with $n<r$. If $q\in[0,3]$ then we even have the stronger statement for every $n\le r^{\{q\}}$ we have $\beta_\Gamma(n)\gg n^{\lfloor q\rfloor+1}\beta_\Gamma(1)$, whilst for every $n$ with $r^{\{q\}}\le n\le r$ we have $\beta_\Gamma(n)\gg r^{\{q\}}n^{\lfloor q\rfloor}\beta_\Gamma(1)$. If $q>3$ then, defining
\[
\begin{array}{ccccl}
h^*&:&(0,\infty)&\to&\N_0\vspace{6pt}\\
   &&     p    &\mapsto&h(\lceil p\rceil-1),
\end{array}
\]
for every $p\in(b(q),q]$ and every $n\in\N$ with
\begin{equation}\label{eq:n.range}
r^\frac{h^*(p)-q}{h^*(p)-p}\le n<r,
\end{equation}
we have
\begin{equation}\label{eq:beta(n)>n^pbeta(1)}
\beta_\Gamma(n)\gg_\fq n^p\beta_\Gamma(1).
\end{equation}
By convention, we define $r^{-\alpha/0}=0$ for $\alpha>0$ in \eqref{eq:n.range}.
\end{prop}
Proposition \ref{prop:growth.lb.rel.all} generalises and refines a result of Breuillard, Green and Tao, who proved the bound \eqref{eq:lb.rel.bgt} with a non-explicit function $b$ and with $\Gamma$ assumed to be a Cayley graph \cite[Corollary 11.10]{bgt}. Our function $b(q)$ grows like $\sqrt{q}$; it is not hard to generalise \cite[Example 1.11]{tao.growth} to an arbitrary filiform group to show that this cannot be improved.

Write $\rad_x(\Gamma)=\min\{r\in\N:B(x,r)=\Gamma\}$ for the \emph{radius} of a conntected graph $\Gamma$ centred at $x\in\Gamma$. Note that if $\Gamma$ is vertex-transitive then $\diam(\Gamma)=\rad_x(\Gamma)$ for every $x\in\Gamma$. It is easy to see that for every connected graph $\Gamma$, every $x\in\Gamma$ and every $r\le\rad_x(\Gamma)$ we have $|B_\Gamma(x,r)|>r$ so one may vacuously extend Proposition \ref{prop:growth.lb} to the case in which $q<1$. For Proposition \ref{prop:growth.lb.rel.all} we have the following complementary result.
\begin{lemma}\label{lem:growth.lb.rel.lin}
Let $\Gamma$ be a connected regular graph of degree $k\in\N$, let $x\in\Gamma$, and fix $n\in\N$ with $n\le\rad_x(\Gamma)$. Then $|B_\Gamma(x,n)|\ge\frac{1}{3}(k+1)n$.
\end{lemma}
\begin{proof}
Since $n\le\rad_x(\Gamma)$ there is a geodesic of length $n$ starting at $x$. Writing $x=x_0,x_1,\ldots,x_n$ for the vertices of this geodesic in increasing order of distance from $x$, the lemma follows from the fact that the balls $B_\Gamma(x_{3m},1)$ with $m\in\N_0$ and $3m<n$ are disjoint subsets of $B_\Gamma(x,n)$ of size $k+1$.
\end{proof}

The main ingredient in the proofs of the theorems of this section is \cref{thm:Growth}, quoted from our companion paper \cite{ttBalls}.

We start by isolating the $q\le3$ cases of Proposition \ref{prop:growth.lb.rel.all}, as follows.
\begin{prop}\label{prop:growth.lb.rel}
Let $q\in[1,3]$, let $r\in\N$, and suppose that $\Gamma$ is a locally finite vertex-transitive graph such that $\beta_\Gamma(r)\ge r^q\beta_\Gamma(1)$. Then for every $n\le r^{\{q\}}$ we have
\begin{equation}\label{eq:first.slope}
\beta_\Gamma(n)\gg n^{\lfloor q\rfloor+1}\beta_\Gamma(1),
\end{equation}
whilst for every $n$ with $r^{\{q\}}\le n\le r$ we have
\begin{equation}\label{eq:second.slope}
\beta_\Gamma(n)\gg r^{\{q\}}n^{\lfloor q\rfloor}\beta_\Gamma(1).
\end{equation}
\end{prop}

\begin{proof}
We start with the case $n\le r^{\{q\}}$, which implies in particular that we may assume that $q\notin\N$, and in particular that $q\ne3$. Since $h(d)=d$ for $d=1,2$, by \cref{thm:tt.rel} there exists $c>0$ such that if $\beta_\Gamma(n)\le\eps n^{\fq+1}\beta_\Gamma(1)$ then $\beta_\Gamma(r)\le\frac{1}{2}r^\fq n\beta_\Gamma(1)$, which is contrary to the hypothesis when $n\le r^{\{q\}}$. It follows that \eqref{eq:first.slope} holds for every such $n$, as required.

If $n\ge r^{\{q\}}$ and
\begin{equation}\label{eq:sec.slope.contra}
\beta_\Gamma(n)\le cr^{\{q\}}n^\fq\beta_\Gamma(1)
\end{equation}
for some $c>0$ then we would have in particular that
\[
\beta_\Gamma(n)\le cn^{\fq+1}\beta_\Gamma(1),
\]
or
\[
\beta_\Gamma(n)\le cn^3\beta_\Gamma(1)
\]
in the case $q=3$. In either case, provided $c$ is small enough, Theorem \ref{thm:tt.rel} would therefore imply that $\beta_\Gamma(r)\ll\left(\frac{r}{n}\right)^\fq\beta_\Gamma(n)$, which combined with \eqref{eq:sec.slope.contra} would imply that $\beta_\Gamma(r)\le\frac{1}{2}r^q\beta_\Gamma(1)$. This would be contrary to the hypothesis, and so \eqref{eq:second.slope} must hold as required.
\end{proof}

\begin{proof}[Proof of Proposition \ref{prop:growth.lb.rel.all}]
For $q\in[0,3]$ the theorem follows immediately from Proposition \ref{prop:growth.lb.rel}, so we may assume that $q>3$ and that
\begin{equation}\label{eq:lb.rel.q=3}
\beta_\Gamma(n)\gg n^3\beta_\Gamma(1)
\end{equation}
for every $n<r$. Rewriting $b$ in terms of $h^*$ instead of $h$, and restricting to the range $q>3$, we have
\begin{equation}\label{eq:rewrite.b}
b(q)=\begin{cases}4 &\text{if $q=4$},\\
    \max\{p\in(0,\infty):h^*(p)\le q\}&\text{otherwise}.
    \end{cases}
\end{equation}

Let $p\in(3,q]$, noting that
\begin{equation}\label{eq:h(p)>p}
h^*(p)\ge p
\end{equation}
in this range, with
\begin{equation}\label{eq:h*(4)}
h^*(4)=4
\end{equation}
the only instance of equality.

Let $n\in\N$ with $n<r$. By Theorem \ref{thm:tt.rel} there exists $c>0$ depending only on $\fq$ such that if
\begin{equation}\label{eq:growth.lb.contra}
\beta_\Gamma(n)\le cn^p\beta_\Gamma(1)
\end{equation}
then $\beta_\Gamma(r)\le\frac{1}{2}r^{h^*(p)}n^{p-h^*(p)}\beta_\Gamma(1)$, and hence by the hypothesis of the corollary that
\begin{equation}\label{eq:lb.before.rearrange}
n^{h^*(p)-p}\le\textstyle\frac{1}{2}r^{h^*(p)-q}.
\end{equation}
By \eqref{eq:h(p)>p}, this implies that
\begin{equation}\label{eq:h(p)>q}
h^*(p)>q.
\end{equation}
If $p=4$ then we would have $p>q$ by \eqref{eq:h*(4)} and \eqref{eq:h(p)>q}, contradicting the choice of $p$,
so it must be the case that
\begin{equation}\label{eq:p=4}
p\ne4.
\end{equation}
Moreover, \eqref{eq:rewrite.b} and \eqref{eq:h(p)>q} imply that
\begin{equation}\label{eq:growth.lb.contrapos.1}
p>b(q).
\end{equation}
Indeed, this is immediate if $q\ne4$, whilst if $q=4$ then \eqref{eq:h(p)>q} implies that $h^*(p)>4$, and it is easy to check, using the definition of $h^*$, that this implies $p>4$, as required. Finally, \eqref{eq:p=4} means in particular that the inequality \eqref{eq:h(p)>p} is strict, and so we may conclude from \eqref{eq:lb.before.rearrange} that
\begin{equation}\label{eq:growth.lb.contrapos.2}
n\le\textstyle\frac{1}{2}r^\frac{h^*(p)-q}{h^*(p)-p}.
\end{equation}

Rephrasing the last paragraph in the contrapositive, we conclude that if any of \eqref{eq:p=4}, \eqref{eq:growth.lb.contrapos.1} or \eqref{eq:growth.lb.contrapos.2} does not hold then neither does \eqref{eq:growth.lb.contra}. In the case of \eqref{eq:growth.lb.contrapos.1}, this means that \eqref{eq:lb.rel.bgt} holds as required. In the case of \eqref{eq:p=4}, this means that if $q\ge4$ and $p=4$ then \eqref{eq:beta(n)>n^pbeta(1)} holds as required. Finally, in the case of \eqref{eq:growth.lb.contrapos.2} this means that if $n$ lies in the range \eqref{eq:n.range} then \eqref{eq:beta(n)>n^pbeta(1)} holds, as required.
\end{proof}

\begin{proof}[Proof of Proposition \ref{prop:growth.lb}]
The proof is identical to that of Proposition \ref{prop:growth.lb.rel}, but with Theorem \ref{thm:tt} \ref{item:Benj.rel} in place of Theorem \ref{thm:tt.rel} \ref{item:growth.abs}. Since we do not use Proposition \ref{prop:growth.lb} to prove our main results, we leave the details to the reader.
\end{proof}

\section{Isoperimetric inequalities in large balls}\label{sec:iso}
In this section we deduce various isoperimetric inequalities for vertex-transitive graphs from the results of the previous two sections. We start by giving uniform bounds in the following classical isoperimetric inequality.
\begin{theorem}\label{thm:iso.classical}
Suppose $\Gamma$ is a vertex-transitive graph of growth at least polynomial of degree $d$. Then for every finite subset $A\subseteq\Gamma$ we have $|\partial A|\gg_d|A|^{\frac{d-1}d}$.
\end{theorem}
\begin{proof}
\cref{thm:Growth} \ref{item:growth.abs} implies that there exists $\eps=\eps(d)$ such that $\beta_\Gamma(n)\ge\eps n^d$ for all $n\in\N$, so the desired result follows from \cref{prop:iso-growth}.
\end{proof}

Our next result refines Theorem \ref{thm:bk.iso.orig}.

\begin{theorem}\label{thm:bk.iso}
Let $q\ge1$, let $\Gamma$ be a connected, locally finite vertex-transitive graph, and let $r\in\N$ be such that $r\le\diam(\Gamma)$. Let $A$ be a finite subset of $\Gamma$ with $|A|\le\beta_\Gamma(r)/2$.
Suppose that $\beta_\Gamma(r)\geq r^q$. Then
\[
|\partial A|\gg_\fq\min\left\{|A|^{\frac{\fq}{\fq+1}},r^\frac{\{q\}}{\fq}|A|^{\frac{\fq-1}{\fq}}\right\}.
\]
\end{theorem}
Theorem \ref{thm:bk.iso} is easily seen to be optimal by considering the case in which $A$ is a ball in the Cayley graph of $\Z^\fq\times\Z/\lceil r^{\{q\}}\rceil\Z$ with respect to the generating set $\{-1,0,1\}^{\fq+1}$.

Our second result is Theorem \ref{thm:isoperim.orig.rel.bgt} with an explicit function $b$.

\begin{theorem}\label{thm:isoperim.orig.rel}
Let $q\ge1$, let $\Gamma$ be a locally finite vertex-transitive graph, let $r\in\N$ be such that $r\le\diam(\Gamma)$, and suppose that $\beta_\Gamma(r)\ge r^q\beta_\Gamma(1)$. Define $b(q)$ as in Proposition \ref{prop:growth.lb.rel.all}, and denote $b=b(q)$. For every subset $A\subset\Gamma$ with $|A|\le\beta_\Gamma(r)/2$ we have
\[
|\partial A|\gg_\fq\beta_\Gamma(1)^\frac{1}{b}|A|^\frac{b-1}{b}.
\]
\end{theorem}

As with Proposition \ref{prop:growth.lb.rel.all}, one can generalise \cite[Example 1.11]{tao.growth} to an arbitrary filiform group to show that the function $b$ is optimal in \cref{thm:isoperim.orig.rel}. That said, we do not quite use every detail of Proposition \ref{prop:growth.lb.rel.all} in proving Theorem \ref{thm:isoperim.orig.rel}, so at the expense of complicating the statement of Theorem \ref{thm:isoperim.orig.rel} somewhat and by using the full strength of Proposition \ref{prop:growth.lb.rel.all} one ought to be able to improve the theorem marginally. We leave this matter to the interested reader.

For the proofs of Theorems \ref{thm:main} and \ref{thm:main.unimod} we will also need the following result.

\begin{theorem}\label{thm:isoperimRel}
Let $q\in[1,3]$, let $\Gamma$ be a connected, locally finite vertex-transitive graph, and let $r\in\N$ be such that $r\le\diam(\Gamma)$. Let $A$ be a finite subset of $\Gamma$ with $|A|\le\beta_\Gamma(r)/2$.
Suppose that $\beta_\Gamma(r)\ge r^q\beta_\Gamma(1)$. Then
\[
|\partial A|\gg\min\left\{\beta_\Gamma(1)^{\frac{1}{\fq+1}}|A|^{\frac{\fq}{\fq+1}},\beta_\Gamma(1)^{\frac{1}{\fq}}r^\frac{\{q\}}{\fq}|A|^{\frac{\fq-1}{\fq}}\right\}.
\]
\end{theorem}

Theorem \ref{thm:isoperimRel} is easily seen to be optimal in general by considering the case in which $A$ is a ball in the Cayley graph of $\Z^\fq\,\times\,(\Z/\lceil r^{\{q\}}\rceil\Z)\,\times\,(\Z/k\Z)$ with respect to the generating set $\{-1,0,1\}^{\fq+1}\times\Z/k\Z$.

\begin{proof}[Proof of Theorems \ref{thm:bk.iso}, \ref{thm:isoperim.orig.rel} and \ref{thm:isoperimRel}]
We start with Theorem \ref{thm:isoperimRel}, since it is that result that is most important for our applications. If $|A|\le\beta_\Gamma(1)/2$ then we trivially have $|\partial A|\ge\beta_\Gamma(1)/2$, and so the proposition holds. Defining $\phi_\Gamma$ as in Proposition \ref{prop:iso-growth}, Proposition \ref{prop:growth.lb.rel} implies that
\[
\phi_\Gamma(m)\ll\max\{\beta_\Gamma(1)^{-\frac{1}{\fq+1}}m^\frac{1}{\fq+1},\beta_\Gamma(1)^{-\frac{1}{\fq}}r^{-\frac{\{q\}}{\fq}}m^\frac{1}{\fq}\}
\]
for every $m\in\N$ with $\beta_\Gamma(1)\le m\le\beta_\Gamma(r)$. By Proposition \ref{prop:iso-growth} the theorem therefore also holds for all sets $A$ with $|A|\ge\beta_\Gamma(1)/2$.

Theorem \ref{thm:bk.iso} follows similarly with Proposition \ref{prop:growth.lb} in place of Proposition \ref{prop:growth.lb.rel}, and Theorem \ref{thm:isoperim.orig.rel} follows similarly with conclusion \eqref{eq:lb.rel.bgt} of Proposition \ref{prop:growth.lb.rel.all}.
\end{proof}

It will be convenient to isolate the following special cases of Theorems \ref{thm:bk.iso} and \ref{thm:isoperimRel}.
\begin{corollary}\label{cor:bk.iso}
Suppose that $\beta_\Gamma(r)= r^q$ in the setting of Theorem \ref{thm:bk.iso}.
\begin{enumerate}[label=(\roman*)]
\item\label{item:bk.iso.main}
We have
\[
|\partial A|\gg_\fq\min\left\{|A|^{\frac{\fq}{\fq+1}},\beta_\Gamma(r)^{\frac{1}{\fq}-\frac{1}{q}}|A|^{1-\frac{1}{\fq}}\right\};
\]
\item\label{item:bk.iso.spec} If $p\ge1$ is such that $\fp=\fq$ then
\[
|\partial A|\gg_\fp\min\left\{|A|^{\frac{\fp}{\fp+1}},|A|^{1-\frac{1}{p}}\left(\frac{\beta_\Gamma(r)}{r^p}\right)^{\frac{1}{p}}\left(\frac{\beta_\Gamma(r)}{|A|}\right)^{\frac{1}{\fp}-\frac{1}{p}}\right\}.
\]
\end{enumerate}
\end{corollary}
\begin{proof}
Conclusion \ref{item:bk.iso.main} follows trivially from substituting $\beta_\Gamma(r)=r^q\beta_\Gamma(1)$ into the conclusion of Theorem \ref{thm:bk.iso}. To check \ref{item:bk.iso.spec}, note that
\[\beta(r)^{\frac{1}{\fp}-\frac{1}{q}}=\beta(r)^{\frac{1}{p}-\frac{1}{q}}\beta(r)^{\frac{1}{\fp}-\frac{1}{p}}=\frac{1}{\beta_\Gamma(1)^{\frac{1}{q}}}\left(\frac{\beta_\Gamma(r)}{r^p}\right)^{\frac{1}{p}}\beta(r)^{\frac{1}{\fp}-\frac{1}{p}}\]
and
\[|A|^{1-\frac{1}{\fp}}=|A|^{1-\frac{1}{p}}\frac{1}{|A|^{\frac{1}{\fp}-\frac{1}{p}}},\]
and substitute these expressions into \ref{item:bk.iso.main}.
\end{proof}
\begin{corollary}\label{cor:isoperimRel}
Suppose that $\beta_\Gamma(r)= r^q\beta_\Gamma(1)$ in the setting of Theorem \ref{thm:isoperimRel}.
\begin{enumerate}[label=(\roman*)]
\item\label{item:isoperimRel.main} We have
\[
|\partial A|\gg\min\left\{\beta_\Gamma(1)^{\frac{1}{\fq+1}}|A|^{\frac{\fq}{\fq+1}},\beta_\Gamma(1)^{\frac{1}{q}}\beta(r)^{\frac{1}{\fq}-\frac{1}{q}}|A|^{1-\frac{1}{\fq}}\right\}.
\]
\item\label{item:isoperimRel.spec} If $p\ge1$ is such that $\fp=\fq$ then
\[
|\partial A|\gg_\fp\min\left\{\beta_\Gamma(1)^{\frac{1}{\fq+1}}|A|^{\frac{\fp}{\fp+1}},|A|^{1-\frac{1}{p}}\left(\frac{\beta_\Gamma(r)}{r^p}\right)^{\frac{1}{p}}\left(\frac{\beta_\Gamma(r)}{|A|}\right)^{\frac{1}{\fp}-\frac{1}{p}}\right\}.
\]
\end{enumerate}
\end{corollary}
\begin{proof}This is almost identical to the proof of Corollary \ref{cor:bk.iso}.
\end{proof}

It will also be useful to have the following inequality, which does not require any growth hypothesis.
\begin{lemma}\label{lem:iso.rel.lin}
Let $r\in\N$ and suppose that $\Gamma$ is a connected, locally finite vertex-transitive graph of diameter at least $r$. Then for every subset $A\subset\Gamma$ with $|A|\le\beta_\Gamma(r)/2$ we have
\[
|\partial A|\ge\frac{\beta_\Gamma(1)}{32}.
\]
\end{lemma}
\begin{proof}If $|A|\le\beta_\Gamma(1)/2$ then we trivially have $|\partial A|\ge\beta_\Gamma(1)/2$, and so the lemma holds. Defining $\phi_\Gamma$ as in Proposition \ref{prop:iso-growth}, Lemma \ref{lem:growth.lb.rel.lin} implies that $\phi_\Gamma(m)\le\lceil3m/\beta_\Gamma(1)\rceil$ for every $m\le\beta_\Gamma(r)$, and in particular that $\phi_\Gamma(m)\le4m/\beta_\Gamma(1)$ for every $m$ satisfying $\beta_\Gamma(1)\le m\le\beta_\Gamma(r)$. Proposition \ref{prop:iso-growth} therefore implies that the lemma holds for every $|A|\ge\beta_\Gamma(1)/2$.
\end{proof}

We close this section with the following converse to Theorem \ref{thm:bk.iso.orig}.
\begin{prop}[converse to Theorem \ref{thm:bk.iso.orig}]\label{prop:iso.conv}
Let $q\ge1$, let $\Gamma$ be a locally finite  vertex-transitive graph, let $e\in\Gamma$, let $r\in\N$ be such that $r\le\diam(\Gamma)$, and suppose that
\begin{equation}\label{eq:iso.conv}
|\partial B_\Gamma(e,n)|\ge\beta_\Gamma(n)^{\frac{q-1}{q}}
\end{equation}
for every $n\in\N$ such that $\beta_\Gamma(n)\le\frac{1}{2}\beta_\Gamma(r)$. Then 
\begin{equation}\label{eq:iso.conv.concl}
\beta_\Gamma(r)\gg_qr^q.
\end{equation}
\end{prop}
\begin{proof}
It is straightforward to check that \eqref{eq:iso.conv.concl} holds if \eqref{eq:iso.conv} holds for every $n=1,\ldots,r$. In particular, if \eqref{eq:iso.conv} holds for every $n=1,\ldots,\lfloor r/4\rfloor$ then $\beta_\Gamma(r)\ge\beta_\Gamma(\lfloor r/4\rfloor)\gg_q\lfloor r/4\rfloor^q\gg_qr^q$, and so \eqref{eq:iso.conv.concl} holds. It therefore suffices to show that $\beta_\Gamma(\lfloor r/4\rfloor)\le\frac{1}{2}\beta_\Gamma(r)$; to see this, note that if $x$ is an arbitrary element at distance, say, $\lfloor 2r/3\rfloor$ from $e$ then $B_\Gamma(e,\lfloor r/4\rfloor)$ and $B_\Gamma(x,\lfloor r/4\rfloor)$ are disjoint subsets of $B_\Gamma(e,r)$ of size $\beta_\Gamma(\lfloor r/4\rfloor)$.
\end{proof}

\section{Upper bounds on $p$-resistance}\label{sec:resist.ub}
In this section we establish the upper bounds on $p$-resistance appearing in our main theorems. We start by considering finite graphs, where the details are slightly cleaner. The following two results imply in particular the upper bounds of Theorems \ref{thm:mainp} and \ref{thm:main.linearp}.

\begin{prop}\label{prop:main.ub}
Let $p\in (1,3)$. Let $\Gamma$ be a finite, connected, vertex-transitive graph of diameter $\gamma$. Then
\[
R_{\Gamma,p}\ll_p\frac{1}{\deg(\Gamma)}+\frac{\gamma^p\left(\log(|\Gamma|/\deg(\Gamma))\right)^{p-1}}{|\Gamma|}.
\]
Moreover, if $p$ is not an integer then 
\[
R_{\Gamma,p}\ll_p\frac{1}{\deg(\Gamma)}+\frac{\gamma^p}{|\Gamma|}.
\]
\end{prop}

\begin{prop}\label{prop:main.p.ub}
Let $p>1$. Let $\Gamma$ be a finite, connected, vertex-transitive graph of diameter $\gamma$, and define $q\ge1$ so that $|\Gamma|=\gamma^q$.
\begin{enumerate}[label=(\roman*)]
\item If $q\ge\fp+1$ then
\[
R_{\Gamma,p}\ll_p\frac{1}{\deg(\Gamma)^{1-\frac{p}{\fp+1}}}.
\]
\item If $\fp\le q<\fp+1$ then
\[
R_{\Gamma,p}\ll_p\frac{1}{\deg(\Gamma)^{1-\frac{p}{\fp+1}}}+\frac{\gamma^p\left(\log(|\Gamma|/\deg(\Gamma))\right)^{p-1}}{|\Gamma|}.
\]
If moreover $p$ is not an integer then
\[
R_{\Gamma,p}\ll_p\frac{1}{\deg(\Gamma)^{1-\frac{p}{\fp+1}}}+\frac{\gamma^p}{|\Gamma|}.
\]
\item If $q\le p$ then
\[
R_{\Gamma,p}\ll_p\frac{1}{\deg(\Gamma)}+\frac{\gamma^p\left(\log(|\Gamma|/\deg(\Gamma))\right)^{p-1}}{|\Gamma|}.
\]
\item If $q<\fp$ and $p$ is not an integer then
\[
R_{\Gamma,p}\ll_p\frac{\gamma^p}{|\Gamma|}.
\]
\item\label{item:main.p.ub.p-eps} If $q\le p-\eps$ for some $\eps>0$ then
\[
R_{\Gamma,p}\ll_{p,\eps}\frac{\gamma^p}{|\Gamma|}.
\]
\end{enumerate}
\end{prop}

Our next two results contain the upper bounds of Theorem \ref{thm:main.unimodp} in the special case in which the diameter of $\Gamma$ is at least $4r$, and a variant of Theorem \ref{thm:main.unimodp.linear}. At the end of the section we show how to deduce Theorems \ref{thm:main.unimodp} and \ref{thm:main.unimodp.linear} in full.

\begin{prop}\label{prop:unimod.ub}
Let $p\in (1,3)$ and $r\in\N$. Let $\Gamma$ be a connected, locally finite vertex-transitive graph of diameter at least $4r$, and let $x\in\Gamma$. Then
\[
R_p(x\leftrightarrow S_\Gamma(x,r+1))\ll_p\frac{1}{\deg(\Gamma)}+\frac{r^p\left(\log(\beta_\Gamma(r)/\deg(\Gamma))\right)^{p-1}}{\beta_\Gamma(r)}.
\]
Moreover, if $p$ is not an integer then 
\[
R_p(x\leftrightarrow S_\Gamma(x,r+1))\ll_p\frac{1}{\deg(\Gamma)}+\frac{r^p}{\beta_\Gamma(r)}.
\]
\end{prop}

\begin{prop}\label{prop:unimod.p.ub}
Let $p>1$ and $r\in\N$. Let $\Gamma$ be a connected, locally finite vertex-transitive graph of diameter at least $4r$, let $x\in\Gamma$, and define $q\ge1$ so that $\beta_\Gamma(4r)=(4r)^q$.
\begin{enumerate}[label=(\roman*)]
\item If $q\ge\fp+1$ then
\[
R_p(x\leftrightarrow S_\Gamma(x,r+1))\ll_p\frac{1}{\deg(\Gamma)^{1-\frac{p}{\fp+1}}}.
\]
\item If $\fp\le q<\fp+1$ then
\[
R_p(x\leftrightarrow S_\Gamma(x,r+1))\ll_p\frac{1}{\deg(\Gamma)^{1-\frac{p}{\fp+1}}}+\frac{r^p\left(\log(\beta_\Gamma(r)/\deg(\Gamma))\right)^{p-1}}{\beta_\Gamma(r)}.
\]
If moreover $p$ is not an integer then
\[
R_p(x\leftrightarrow S_\Gamma(x,r+1))\ll_p\frac{1}{\deg(\Gamma)^{1-\frac{p}{\fp+1}}}+\frac{r^p}{\beta_\Gamma(r)}.
\]
\item If $q\le p$ then
\[
R_p(x\leftrightarrow S_\Gamma(x,r+1))\ll_p\frac{1}{\deg(\Gamma)}+\frac{r^p\left(\log(\beta_\Gamma(r)/\deg(\Gamma))\right)^{p-1}}{\beta_\Gamma(r)}.
\]
\item If $q<\fp$ and $p$ is not an integer then
\[
R_p(x\leftrightarrow S_\Gamma(x,r+1))\ll_p\frac{r^p}{\beta_\Gamma(r)}.
\]
\item\label{item:main.p.ub.p-eps} If $q\le p-\eps$ for some $\eps>0$ then
\[
R_p(x\leftrightarrow S_\Gamma(x,r+1))\ll_{p,\eps}\frac{r^p}{\beta_\Gamma(r)}.
\]
\end{enumerate}
\end{prop}

We prove these propositions by using the isoperimetric inequalities of Section \ref{sec:iso} to bound the quantities $j_{A,p}$ appearing in Theorems \ref{thm:bk} and \ref{thm:bk.inf}. In fact, in our arguments it is only the $|A|/|\partial A|^{p/(p-1)}$ term of $j_{A,p}$ that will be relevant; we capture this with the following lemma.

\begin{lemma}\label{lem:secondtermonly}
Let $C>0$ and $p>1$. Let $\Gamma$ be a graph. Suppose that $A\subset\Gamma$ is a finite set of vertices and $\xi\le C|A|$ is such that $|\partial A|\ge\xi$. Then
\[
j_{A,p}\le\frac{(1+C)|A|}{\xi^\frac{p}{p-1}}.
\]
\end{lemma}
\begin{proof}
We have
\begin{align*}
j_{A,p}&\le\frac{|A|}{|\partial A|^{\frac{p}{p-1}}}+\frac{1}{|\partial A|^{\frac{1}{p-1}}}\\
   &\le\frac{|A|}{|\partial A|^{\frac{p}{p-1}}}+\frac{C|A|}{\xi|\partial A|^{\frac{1}{p-1}}}\\
   &\le\frac{(1+C)|A|}{\xi^\frac{p}{p-1}},
\end{align*}
as required.
\end{proof}

We now collect together the necessary bounds on $j_{A,p}$.
\begin{lemma}\label{lem:j.ub.p=3}
Let $p\in(1,3)$. Let $\Gamma$ be a connected, locally finite vertex-transitive graph. Let $r\in\N$ be such that $r\le\diam(\Gamma)$, and define $q\ge0$ so that $\beta_\Gamma(r)=r^q\beta_\Gamma(1)$. Set $\eta=\frac{p}{\fp+1}$.
\begin{enumerate}[label=(\roman*)]
\item If $q\ge\fp+1$ then
\[
j_{A,p}\ll_p\frac{1}{\deg(\Gamma)^{\frac{\eta}{p-1}}|A|^{\frac{1-\eta}{p-1}}}
\]
for every $A\subset\Gamma$ with $\deg(\Gamma)/2\le|A|\le\beta_\Gamma(r)/2$.
\item If $\fp\le q<\fp+1$ then
\[
j_{A,p}\ll_p\frac{1}{\deg(\Gamma)^{\frac{\eta}{p-1}} |A|^{\frac{1-\eta}{p-1}}}+\left(\frac{r^p}{\beta_\Gamma(r)}\right)^{\frac{1}{p-1}}\left(\frac{|A|}{\beta_\Gamma(r)}\right)^{\frac{p-\fp}{\fp(p-1)}}
\]
for every $A\subset\Gamma$ with $\deg(\Gamma)/2\le|A|\le\beta_\Gamma(r)/2$.
\item If $1\le q\le\fp$ then
\[
j_{A,p}\ll_p \frac{|A|^{\frac{p-q}{(p-1)q}}}{\deg(\Gamma)^{\frac{p}{(p-1)q}}}
\]
for every $A\subset\Gamma$ with $\deg(\Gamma)/2\le|A|\le\beta_\Gamma(r)/2$.
\item\label{item:j.ub.p=3.q<1} Regardless of the value of $q$, we have
\[
j_{A,p}\ll\frac{|A|}{\deg(\Gamma)^\frac{p}{p-1}}
\]
for every $A\subset\Gamma$ with $\deg(\Gamma)/2\le|A|\le\beta_\Gamma(r)/2$.
\end{enumerate}
\end{lemma}
\begin{proof}
Define the function $b$ as in Proposition \ref{prop:growth.lb.rel.all}.
\begin{enumerate}[label=(\roman*)]
\item\label{item:j.ub.p=3} Since $b(\fp+1)=\fp+1$ for $p<3$, Theorem \ref{thm:isoperim.orig.rel} and Lemma \ref{lem:secondtermonly} imply that
\[
j_{A,p}\ll_p \frac{|A|}{\left(\deg(\Gamma)^{\frac{1}{\fp+1}}|A|^{\frac{\fp}{\fp+1}}\right)^{\frac{p}{p-1}}}= \frac{1}{\deg(\Gamma)^{\frac{\eta}{p-1}} |A|^{\frac{1-\eta}{p-1}}},
\]
as required.
\item This follows from Corollary \ref{cor:isoperimRel} \ref{item:isoperimRel.spec} and Lemma \ref{lem:secondtermonly}, using the same calculation as in the proof of \ref{item:j.ub.p=3}.
\item Since $b(q)=q$ for this value of $q$, the desired bound follows from Theorem \ref{thm:isoperim.orig.rel} and Lemma \ref{lem:secondtermonly}.
\item This is immediate from Lemmas \ref{lem:iso.rel.lin} and \ref{lem:secondtermonly}.
\end{enumerate}
\end{proof}

\begin{lemma}\label{lem:j.ub.pq}
Let $p>1$. Let $\Gamma$ be a connected, locally finite vertex-transitive graph. Let $r\in\N$ be such that $r\le\diam(\Gamma)$, let $q\ge1$, and suppose that $\beta_\Gamma(r)\ge r^q$. Then
\[
j_{A,p}\ll_{p,\fq}|A|^\frac{p-q}{q(p-1)}
\]
for every $A\subset\Gamma$ with $|A|\le\beta_\Gamma(r)/2$.
\end{lemma}
\begin{proof}
\cref{thm:bk.iso.orig} implies that $|\partial A|\gg_\fq|A|^{\frac{q-1}{q}}$ for every $A\subset\Gamma$ with $|A|\le\beta_\Gamma(r)/2$, and so the lemma follows from Lemma \ref{lem:secondtermonly}.
\end{proof}

\begin{lemma}\label{lem:j.ub.fp=fq}
Let $p>1$. Let $\Gamma$ be a connected, locally finite vertex-transitive graph. Let $r\in\N$ be such that $r\le\diam(\Gamma)$, and define $q\ge1$ such that $\beta_\Gamma(r)=r^q$. Suppose that $\fp\le q<\fp+1$. Then
\[
j_{A,p}\ll_p\frac{1}{|A|^{\frac{\fp+1-p}{(\fp+1)(p-1)}}}+ \left(\frac{r^p}{\beta_\Gamma(r)}\right)^{\frac{1}{p-1}}\left(\frac{|A|}{\beta_\Gamma(r)}\right)^{\frac{p-\fp}{(p-1)\fp}}.
\]
for every $A\subset\Gamma$ with $|A|\le\beta_\Gamma(r)/2$.
\end{lemma}
\begin{proof}
This follows from Corollary \ref{cor:bk.iso} \ref{item:bk.iso.spec} and Lemma \ref{lem:secondtermonly}.
\end{proof}

We now prove the main results of this section.
\begin{proof}[Proof of Proposition \ref{prop:main.ub}]
Set $q\ge0$ so that $|\Gamma|=\gamma^q\beta_\Gamma(1)$, noting that
\begin{equation}\label{eq:main.ub.gamma/|G|}
\left(\frac{|\Gamma|}{\deg(\Gamma)}\right)^\frac pq\ll_p\left(\frac{|\Gamma|}{\beta_\Gamma(1)}\right)^\frac pq=\gamma^p.
\end{equation}
If $q\ge\fp+1$ then Lemma \ref{lem:j.ub.p=3} and Theorem \ref{thm:bk} imply that $R_{\Gamma,p}\ll_p1/\deg(\Gamma)$, and the proposition is satisfied. If $\fp\le q<\fp+1$ then the desired bounds follow from Lemma \ref{lem:j.ub.p=3}, Theorem \ref{thm:bk} and Lemma \ref{lem:^p.dist}. If $1\le q<\fp$ then Lemma \ref{lem:j.ub.p=3}, Theorem \ref{thm:bk} and Lemma \ref{lem:^p.dist} imply that
\[
R_{\Gamma,p}\ll_p\frac{1}{\deg(\Gamma)}+\frac{|\Gamma|^{\frac{p-q}{q}}\left(\log(|\Gamma|/\deg(\Gamma))\right)^{p-1}}{\deg(\Gamma)^{\frac{p}{q}}},
\]
and even
\[
R_{\Gamma,p}\ll_p\frac{1}{\deg(\Gamma)}+\frac{|\Gamma|^{\frac{p-q}{q}}}{\deg(\Gamma)^{\frac{p}{q}}}
\]
if $p$ is not an integer. In each case the desired bounds follow from \eqref{eq:main.ub.gamma/|G|}. Finally, Lemma \ref{lem:j.ub.p=3} \ref{item:j.ub.p=3.q<1} and Theorem \ref{thm:bk} imply that
\[
R_{\Gamma,p}\ll_p\frac{|\Gamma|^{p-1}}{\deg(\Gamma)^p},
\]
which if $q<1$ gives
\[
R_{\Gamma,p}\ll_p\frac{\gamma^p}{|\Gamma|},
\]
as required.
\end{proof}

\begin{prop}\label{prop:Rpleqq'}
Let $p>1$. Let $\Gamma$ be a finite, connected vertex-transitive graph of diameter $\gamma$. Let $q\ge1$, and suppose that $|\Gamma|\ge\gamma^q$. Then
\begin{equation}\label{eqR_p'}
R_{\Gamma,p}\ll_{p,\fq}\frac{1}{\deg(\Gamma)}+\left(\sum_{n=\lfloor\log_2\beta_{\Gamma}(1)\rfloor}^{\lfloor\log_2|\Gamma|\rfloor}  2^{\frac{(p-q)n}{q(p-1)}}\right)^{p-1}
\end{equation}
In particular,
if $q\ne p$ then
\[
R_{\Gamma,p}\ll_{p,\fq}\frac{1}{\deg(\Gamma)}+\left(\frac{\deg(\Gamma)^{\frac{p-q}{q(p-1)}}-|\Gamma|^{\frac{p-q}{q(p-1)}}}{1-2^{\frac{p-q}{q(p-1)}}}\right)^{p-1},
\]
and if $q=p$ then 
\[
R_{\Gamma,p}\ll_p\left(\log(|\Gamma|/\deg(\Gamma))\right)^{p-1}.
\]
\end{prop}
\begin{proof}
This follows from Lemma \ref{lem:j.ub.pq}, Theorem \ref{thm:bk} and Lemma \ref{lem:^p.dist}.
\end{proof}

\begin{proof}[Proof of Proposition \ref{prop:main.p.ub}]
First, note that by definition of $q$ we have
\begin{equation}\label{eq:Gamma^((p-q)/q)}
|\Gamma|^\frac{p-q}{q}=\gamma^p/|\Gamma|.
\end{equation}
We now consider the different parts of the proposition in turn. 
\begin{enumerate}[label=(\roman*)]
\item Since $|\Gamma|\ge\gamma^{\fp+1}$, the required bound follows immediately from the $p\ne q$ case of Proposition \ref{prop:Rpleqq'}.
\item This follows from Lemma \ref{lem:j.ub.fp=fq}, Theorem \ref{thm:bk} and Lemma \ref{lem:^p.dist}.
\item Every term of the sum in \eqref{eqR_p'} is at most 
$|\Gamma|^{\frac{p-q}{q(p-1)}}$, so Proposition \ref{prop:Rpleqq'} gives
\[
R_{\Gamma,p} \ll_p \frac{1}{\deg(\Gamma)}+  \left(\log(|\Gamma|/\deg(\Gamma))\right)^{p-1} |\Gamma|^{\frac{p-q}{q}}
\]
and the desired bound follows from \eqref{eq:Gamma^((p-q)/q)}.
\item This is immediate from \ref{item:main.p.ub.p-eps}, which we are about to prove.
\item The case $p\ne q$ of \cref{prop:Rpleqq'} gives
\[
R_{\Gamma,p}\ll_{p}\frac{|\Gamma|^{\frac{p-q}{q}}}{\left(2^{\frac{p-q}{q(p-1)}}-1\right)^{p-1}},
\]
which if $q<p-\eps$ implies that
\[
R_{\Gamma,p}\ll_{p,\eps} |\Gamma|^{\frac{p-q}{q}},
\]
and so the desired bound follows from \eqref{eq:Gamma^((p-q)/q)}.
\end{enumerate}
\end{proof}

When considering graphs that are not necessarily finite, it will be useful to have the following simple lemma, which allows us to prove isoperimetric inequalities for subsets of $B_\Gamma(x,r)$ given lower bounds on $\beta_\Gamma(4r)$.

\begin{lemma}\label{lem:4r>2.r}
Let $r\in\N$ and let $\Gamma$ be a locally finite vertex-transitive graph of diameter at least $4r$. Then $\beta_\Gamma(r)\le\frac12\beta_\Gamma(4r)$.
\end{lemma}
\begin{proof}
Let $x,y\in\Gamma$ be such that $d(x,y)=3r$, and note that $B_\Gamma(x,r)$ and $B_\Gamma(y,r)$ are disjoint subsets of $B_\Gamma(x,4r)$.
\end{proof}

\begin{proof}[Proof of Proposition \ref{prop:unimod.ub}]
Set $q\ge0$ so that $\beta_\Gamma(4r)=(4r)^q\beta_\Gamma(1)$, noting that
\begin{equation}\label{eq:unimod.ub}
\left(\frac{\beta_\Gamma(r)}{\deg(\Gamma)}\right)^\frac pq\ll_p\left(\frac{\beta_\Gamma(r)}{\beta_\Gamma(1)}\right)^\frac pq=r^p.
\end{equation}
If $q\ge\fp+1$ then Lemmas \ref{lem:j.ub.p=3} and \ref{lem:4r>2.r} and Theorem \ref{thm:bk.inf} imply that $R_p(x\leftrightarrow S_\Gamma(x,r+1))\ll_p1/\deg(\Gamma)$, and the proposition is satisfied. If $\fp\le q<\fp+1$ then the desired bounds follow from Lemma \ref{lem:j.ub.p=3}, Lemma \ref{lem:4r>2.r}, Theorem \ref{thm:bk.inf} and Lemma \ref{lem:^p.dist}. If $1\le q<\fp$ then Lemma \ref{lem:j.ub.p=3}, Lemma \ref{lem:4r>2.r}, Theorem \ref{thm:bk.inf} and Lemma \ref{lem:^p.dist} imply that
\[
R_p(x\leftrightarrow S_\Gamma(x,r+1))\ll_p\frac{1}{\deg(\Gamma)}+\frac{\beta_\Gamma(r)^{\frac{p-q}{q}}\left(\log(\beta_\Gamma(r)/\deg(\Gamma))\right)^{p-1}}{\deg(\Gamma)^{\frac{p}{q}}},
\]
and even
\[
R_p(x\leftrightarrow S_\Gamma(x,r+1))\ll_p\frac{1}{\deg(\Gamma)}+\frac{\beta_\Gamma(r)^{\frac{p-q}{q}}}{\deg(\Gamma)^{\frac{p}{q}}}
\]
if $p$ is not an integer. In each case the desired bounds follow from \eqref{eq:unimod.ub}. Finally, Lemma \ref{lem:j.ub.p=3} \ref{item:j.ub.p=3.q<1}, Lemma \ref{lem:4r>2.r} and Theorem \ref{thm:bk.inf} imply that
\[
R_p(x\leftrightarrow S_\Gamma(x,r+1))\ll_p\frac{\beta_\Gamma(r)^{p-1}}{\deg(\Gamma)^p},
\]
which if $q<1$ gives
\[
R_p(x\leftrightarrow S_\Gamma(x,r+1))\ll_p\frac{r^p}{\beta_\Gamma(r)},
\]
as required.
\end{proof}

The following proposition plays a role analogous to that played by Proposition \ref{prop:Rpleqq'} in the case of finite graphs.

\begin{prop}\label{prop:Rpleqq}
Let $p>1$ and $q\ge1$. Let $\Gamma$ be a connected, locally finite vertex-transitive graph, let $r\in\N$ be such that $\diam(\Gamma)\ge4r$, and suppose that $\beta_\Gamma(4r)\ge(4r)^q$. Then
\begin{equation}\label{eqR_p}
R_p(x\leftrightarrow S_\Gamma(x,r+1))\ll_{p,\fq}\frac{1}{\deg(\Gamma)}+\left(\sum_{n=\lfloor\log_2\beta_{\Gamma}(1)\rfloor}^{\lfloor\log_2\beta_\Gamma(r)\rfloor}2^{\frac{(p-q)n}{q(p-1)}}\right)^{p-1}
\end{equation}
In particular,
if $q\ne p$ then
\[
R_p(x\leftrightarrow S_\Gamma(x,r+1))\ll_{p,\fq}\frac{1}{\deg(\Gamma)}+\left(\frac{\deg(\Gamma)^{\frac{p-q}{q(p-1)}}-\beta_\Gamma(r)^{\frac{p-q}{q(p-1)}}}{1-2^{\frac{p-q}{q(p-1)}}}\right)^{p-1},
\]
and if $q=p$ then 
\[
R_p(x\leftrightarrow S_\Gamma(x,r+1))\ll_p\left(\log(\beta_\Gamma(r)/\deg(\Gamma))\right)^{p-1}.
\]
\end{prop}
\begin{proof}
This follows from Lemma \ref{lem:j.ub.pq}, Lemma \ref{lem:4r>2.r}, Theorem \ref{thm:bk.inf} and Lemma \ref{lem:^p.dist}.
\end{proof}

\begin{proof}[Proof of Proposition \ref{prop:unimod.p.ub}]
First, note that
\begin{equation}\label{eq:Gamma^((p-q)/q).inf}
\beta_\Gamma(r)^\frac{p-q}{q}\le\beta_\Gamma(4r)^\frac{p}{q}\beta_\Gamma(r)^{-1}\ll_p\frac{r^p}{\beta_\Gamma(r)}
\end{equation}
by definition of $q$, and that
\begin{equation}\label{eq:4r->r}
\frac{(4r)^p}{\beta_\Gamma(4r)}\ll_p\frac{r^p}{\beta_\Gamma(4r)}\le\frac{r^p}{\beta_\Gamma(r)}.
\end{equation}
We now consider the different parts of the proposition in turn. 
\begin{enumerate}[label=(\roman*)]
\item Since $\beta_\Gamma(4r)\ge(4r)^{\fp+1}$, the required bound follows immediately from the $p\ne q$ case of Proposition \ref{prop:Rpleqq}.
\item This follows from Lemma \ref{lem:j.ub.fp=fq}, Lemma \ref{lem:4r>2.r}, Theorem \ref{thm:bk.inf}, Lemma \ref{lem:^p.dist}, and \eqref{eq:4r->r}.
\item Every term of the sum in \eqref{eqR_p} is at most 
$\beta_\Gamma(r)^{\frac{p-q}{q(p-1)}}$, so Proposition \ref{prop:Rpleqq} gives
\[
R_p(x\leftrightarrow S_\Gamma(x,r+1))\ll_p \frac{1}{\deg(\Gamma)}+\left(\log(\beta_\Gamma(r)/\deg(\Gamma))\right)^{p-1}\beta_\Gamma(r)^{\frac{p-q}{q}}
\]
and the desired bound follows from \eqref{eq:Gamma^((p-q)/q).inf}.
\item This is immediate from \ref{item:main.p.ub.p-eps}, which we are about to prove.
\item The case $p\ne q$ of \cref{prop:Rpleqq} gives
\[
R_{\Gamma,p}\ll_{p}\frac{\beta_\Gamma(r)^{\frac{p-q}{q}}}{\left(2^{\frac{p-q}{q(p-1)}}-1\right)^{p-1}},
\]
which if $q<p-\eps$ implies that
\[
R_{\Gamma,p}\ll_{p,\eps}\beta_\Gamma(r)^{\frac{p-q}{q}},
\]
and so the desired bound follows from \eqref{eq:Gamma^((p-q)/q).inf}.
\end{enumerate}
\end{proof}

\begin{proof}[Proof of Theorem \ref{thm:main.unimodp}]
If $\diam(\Gamma)\ge4r$ then the theorem follows from Propositions \ref{prop:unimod.ub} and \ref{prop:unimod.p.ub}. If not then we have $\diam(\Gamma)^p\le4^pr^p$, so since $R_p(x\leftrightarrow S_\Gamma(x,r+1))\le R_{\Gamma,p}$ the desired bounds follow from Theorem \ref{thm:mainp}.
\end{proof}

\begin{proof}[Proof of Theorem \ref{thm:main.unimodp.linear}]
Given $p>1$ and $R\in\N$ there are only finitely many possibilities for a ball of size at most $R^p$ in a graph, so we upon adjusting the implied constants if necessary we may assume that $r$ is sufficiently large in terms of $p$ and $\eps$. Theorem \ref{thm:tt} then implies that
\begin{equation}\label{eq:main.linear.4r}
\beta_\Gamma(m)\le m^{p-\eps/2}
\end{equation}
for every $m\ge r$. If $\diam(\Gamma)\ge4r$, the theorem then follows from Proposition \ref{prop:unimod.p.ub} \ref{item:main.p.ub.p-eps} and the $m=4r$ case of \eqref{eq:main.linear.4r}. If not then we have $\diam(\Gamma)^p\le4^pr^p$, so since $R_p(x\leftrightarrow S_\Gamma(x,r+1))\le R_{\Gamma,p}$ the desired bound therefore follows from Theorem \ref{thm:main.linearp} and the $m=\diam(\Gamma)$ case of \eqref{eq:main.linear.4r}.
\end{proof}

\section{The Nash-Williams inequality and lower bounds on $p$-resistance}\label{sec:nash-w}

A useful tool for giving lower bounds on resistance is the \emph{Nash-Williams inequality}.
Given two disjoint subsets of vertices $U$ and $U'$, a set $\Pi$ of edges is said to {\it separate} $U$ from $U'$ if every path joining a vertex $u\in U$ to a vertex $v\in V$ must pass through an edge of $\Pi$. The Nash-Williams inequality for $p$-resistance is then as follows.
\begin{prop}\label{prop:Nash-William}
Let $\Pi_1, \ldots, \Pi_k$ be disjoint subsets of edges, each of which separates $U$ from $U'$.
Then
\[
R_p(U\leftrightarrow U')\geq \left(\sum_{i=1}^k\frac{1}{|\Pi_i|}^{\frac{1}{p-1}}\right)^{p-1}.
\]
\end{prop}
\begin{proof}
Let $f$ be a $p$-potential between $U$ and $U'$ such that $C_p(f)=1$. For each $i$, let $U_i$ be the set vertices of all of those connected components of $\Gamma\setminus\Pi_i$ that contain at least one element of $U$.
Observe that $\partial U_i$ is contained in $\Pi_i$. For every $e\in \partial U_i$, let $\bar{e}=(x_e,y_e)$ be the corresponding outward oriented edge. 
Then by (\ref{eq:stokes}), we have for every $i=1,\ldots, k$,
\[1=C_p(f)=\sum_{e\in \partial U_i}\nabla_pf(\bar{e})\leq \sum_{e\in \Pi_i}|f(y_e)-f(x_e)|^{p-1}.\]
By H\"older's inequality, we have that
\[1\leq |\Pi_i|^{\frac{1}{p-1}}\sum_{e\in \Pi_i}|f(y_e)-f(x_e)|^{p}.\]
Summing over $i$, we get
\[\EE_p(f)\geq \sum_{i=1}^k\frac{1}{|\Pi_i|}^{\frac{1}{p-1}},\]
and we conclude by Proposition \ref{prop:p-energy/p-resistance}.
\end{proof}

In this section we use the Nash-Williams inequality to prove the resistance lower bounds of our main theorems. For the convenience of the reader we isolate the following easy lemma.
\begin{lemma}\label{lem:^p.dist}
For every $a,b,p\ge0$ we have $(a+b)^p\asymp_pa^p+b^p$.
\end{lemma}
\begin{proof}
We have $a^p+b^p\asymp\max\{a^p,b^p\}=\max\{a,b\}^p\asymp_p(a+b)^p$.
\end{proof}

\begin{proof}[Proof of Theorem \ref{thm:main.unimodp} (lower bound)]
For each $i=0,\ldots,r$, let $X_i$ be the edge boundary of $B_{\Gamma}(x,i)$, noting that the $X_i$ are disjoint and separate $x$ from $\Gamma\setminus B_{\Gamma}(x,r)$. We have $|X_0|=\deg(\Gamma)$ by definition, and at least $r/2$ of the sets $X_1,\ldots,X_r$ have size at most $2\deg(\Gamma)(\beta_\Gamma(r)-1)/r$. Proposition \ref{prop:Nash-William} therefore implies that
\[
R_p(x\leftrightarrow\Gamma\setminus B_\Gamma(x,r))^\frac{1}{p-1}\gg_p\left(\frac{1}{\deg(\Gamma)}\right)^\frac{1}{p-1}+r\left(\frac{r}{\deg(\Gamma)\beta_\Gamma(r)}\right)^\frac{1}{p-1},
\]
and so the lower bound of Theorem \ref{thm:main.unimodp} follows from Lemma \ref{lem:^p.dist}
\end{proof}

The lower bound of Theorem \ref{thm:mainp} is a special case of the following more general bound.
\begin{prop}\label{prop:lower.bound}
Let $\Gamma$ be a finite graph, and let $u,v\in\Gamma$ be such that $d(u,v)=\diam(\Gamma)$. Then
\[
R_p(u \leftrightarrow  v)^{\frac{1}{p-1}}\ge\left(\frac{1}{\deg(u)}\right)^{\frac{1}{p-1}}+\left(\frac{(\diam(\Gamma)-1)^p}{4(|E(\Gamma)|-\deg(u))}\right)^{\frac{1}{p-1}}.
\]
In particular,
\[
R_p(u \leftrightarrow  v)^\frac{1}{p-1}\ge\left(\frac{1}{\deg(u)}\right)^{\frac{1}{p-1}}+\left(\frac{(\diam(\Gamma)-1)^p}{2\deg(\Gamma)(|\Gamma|-2)}\right)^{\frac{1}{p-1}}.
\]
\end{prop}
\begin{proof}
If $\Gamma$ is not connected then both sides of the inequality are infinite and the proposition holds, so we may assume that $\Gamma$ is connected. Write $\gamma=\diam(\Gamma)$. For each $i=0,\ldots,\gamma-1$, let $X_i$ be the edge boundary of $B_\Gamma(u,i)$, noting that the $X_i$ are disjoint cutsets for $u$ and $v$. We have $|X_0|=\deg(u)$ by definition. Moreover, at least $\frac{1}{2}(\gamma-1)$ of the sets $X_1,\ldots,X_{\gamma-1}$ have size at most
\[
\frac{2(|E(\Gamma)|-\deg(u))}{\gamma-1}.
\]
The proposition therefore follows from the Nash-Williams inequality (Proposition \ref{prop:Nash-William}).
\end{proof}

\begin{proof}[Proof of Theorem \ref{thm:mainp} (lower bound)]
This follows from Lemma \ref{lem:^p.dist} and the second conclusion of Proposition \ref{prop:lower.bound}.
\end{proof}

\begin{proof}[Proof of Theorem \ref{thm:var.converse}]
Theorem \ref{thm:tt.rel} implies that if \eqref{var.conv.hyp} holds for large enough $n$ then
\begin{equation}\label{eq:var.conv.gr.bound}
\beta_\Gamma(m)\ll\left(\frac{m}{n}\right)^2\beta_\Gamma(n)
\end{equation}
for every $m\ge n$.

We start by proving the theorem in the case in which $n<r\le4n$. In that case, defining $\alpha\in(0,3]$ so that $r=(1+\alpha)n$, the bound \eqref{eq:var.conv.gr.bound} implies that $\beta_\Gamma(r)\le(1+\alpha)^2\beta_\Gamma(n)\le(1+5\alpha)\beta_\Gamma(n)$. This implies in particular that for at least half of the values of $k\in\{n+1,\ldots r\}$ we have
\[
\sigma_\Gamma(k)\le\frac{10\beta_\Gamma(n)}{n},
\]
and hence
\[
|\partial^E B(x,k-1)|\le\frac{10\deg(\Gamma)\beta_\Gamma(n)}{n}.
\]
Proposition \ref{prop:Nash-William} therefore implies that 
\begin{align*}
R[S_\Gamma(x,n)\leftrightarrow S_\Gamma(x,r)]&\ge\frac{\alpha n^2}{10\deg(\Gamma)\beta_\Gamma(n)}\\
     &=\frac{n^2}{10\deg(\Gamma)\beta_\Gamma(n)}\left(\frac{r}{n}-1\right)\\
     &\ge\frac{n^2}{10\deg(\Gamma)\beta_\Gamma(n)}\log\frac{r}{n},
\end{align*}
proving the theorem.

We now consider the case in which $r>4n$. The bound \eqref{eq:var.conv.gr.bound} implies that, for every $m\ge2n$, we have
\[
\sigma_\Gamma(k)\ll\frac{\beta_\Gamma(n)k}{n^2}
\]
for at least half of all values of $k\in\{2n,\ldots,m\}$; indeed, given $c>0$, if $\sigma_\Gamma(k)\ge c\beta_\Gamma(n)k/n^2$ for more than half of all $k\in\{2n,\ldots,m\}$ then we would have
\begin{align*}
\beta_\Gamma(m)&\ge\sum_{k=2n}^{\lceil (m+2n)/2\rceil}\frac{c\beta_\Gamma(n)k}{n^2}\\
   &\gg\frac{cm^2\beta_\Gamma(n)}{n^2},
\end{align*}
contradicting \eqref{eq:var.conv.gr.bound} for large enough $c$. This implies in particular that, for every $m\ge2n$, we have
\[
|\partial^E B(x,k)|\ll\frac{\deg(\Gamma)\beta_\Gamma(n)k}{n^2}
\]
for at least half of all values of $k\in\{2n,\ldots,m\}$. In particular, given $r\ge4n$ we have
\begin{align*}
\sum_{k=2n}^r\frac{1}{\sigma_\Gamma(k)}&\gg\frac{n^2}{\deg(\Gamma)\beta_\Gamma(n)}\sum_{i=n}^{\lfloor r/2\rfloor}\frac{1}{2i}\\
   &\gg\frac{n^2}{\deg(\Gamma)\beta_\Gamma(n)}(\log r-\log n)
\end{align*}
and so the theorem follows from Proposition \ref{prop:Nash-William}.
\end{proof}

\section{Tightness of the bounds in our main theorems}\label{sec:bk.examples}
In this section we give examples to show that the bounds in our main theorems are tight up to the multiplicative constants. We start with our upper bounds on $p$-resistance.
\begin{prop}[sharpness of our upper bounds on $p$-resistance]\label{prop:upper.sharp}
Let $p>1$, and fix $n,k,d\in\N$, with $n$ even. Let $G=(\Z/n\Z)^d\times(\Z/k\Z)$, and set
\[
S=\Big(\{-1,0,1\}^d\times\{0\}\Big)\cup\Big(\{0\}^d\times(\Z/k\Z)\Big).
\]
Then
\[
R_p(0\leftrightarrow G\setminus B_S(0,\textstyle{\frac{n}{2}}))\gg_{p,d}
\begin{cases}
n^{p-d}/k&\text{if $d<p$},\\
(\log n)^{p-1}/k&\text{if $d=p$},\\
1/k&\text{if $d>p$}
\end{cases}
\]
and
\[
R_{\Gamma,p}\gg_{p,d}
\begin{cases}
n^{p-d}/k&\text{if $d<p$},\\
(\log n)^{p-1}/k&\text{if $d=p$},\\
1/k&\text{if $d>p$}.
\end{cases}
\]
\end{prop}
See Table \ref{tab:examples} for the values one can take for the parameters in Proposition \ref{prop:upper.sharp} in order to achieve the upper bounds on $p$-resistance in our main theorems.

\begin{proof}
For each $i=0,1,\ldots,\frac{n}{2}$, let $X_i$ be the edge boundary of $B_S(u,i)$. Note that the sets $X_i$ are disjoint and separate $0$ from $G\setminus B_S(0,\frac{n}{2})$, and that $|X_i|\ll_di^{d-1}k$ for each $i$. It therefore follows from the Nash-Williams inequality (Proposition \ref{prop:Nash-William}) that
\[
R_p(0\leftrightarrow G\setminus B_S(0,\textstyle\frac{n}{2}))\ge\displaystyle\frac{1}{k}\left(\sum_{i=1}^{n/2}i^{(1-d)/(p-1)}\right)^{p-1},
\]
and also that $R_{\Gamma,p}$ is bounded below by the same quantity, giving the desired conclusion.
\end{proof}

\begin{table}[t]
\begin{center}
  \begin{tabular}{ c | c }
    Theorem & Parameters \\ \hline
    \ref{thm:main.unimod} & $p=2$; $d=2,3$\\
    \ref{thm:main} & $p=2$; $d=2,3$\\
    \ref{thm:main.unimodp} & $k=1$\\
    \ref{thm:mainp} & $k=1$\\
    \ref{thm:main.unimodp.linear} & $d=p-1$; $k=n^{1-\eps}$\\
    \ref{thm:main.linearp} & $d=p-1$; $k=n^{1-\eps}$
  \end{tabular}
\end{center}
\caption{Values of the parameters in Proposition \ref{prop:upper.sharp} that achieve the upper bounds on $p$-resistance in our main theorems.}\label{tab:examples}
\end{table}

We now give various examples to show that the lower bounds of Theorems \ref{thm:main.unimodp} and \ref{thm:mainp} are sharp. The we start by showing that the lower bounds of Theorems \ref{thm:main.unimodp} and \ref{thm:mainp} are optimal when $p\in\N$ and $\beta_\Gamma(r)\asymp r^p$ or $|\Gamma|\asymp\diam(\Gamma)^p$.
\begin{prop}\label{prop:resist.lb.ex}
Let $r,p\in\N$ be such that $r^{(p-1)/p}$ is an integer, and let $k\in\N$. Define $\Gamma_\infty$ to be the Cayley graph
\[
\Gamma_\infty=\Cay(\,\Z\times(\Z/r^{(p-1)/p}\Z)^p\times\Z/k\Z\,\,,\,\{-1,0,1\}^{p+1}\times\Z/k\Z\,),
\]
and define $\Gamma_0$ to be the Cayley graph
\[
\Gamma_0=\Cay(\,\Z/r\Z\times(\Z/r^{(p-1)/p}\Z)^p\times\Z/k\Z\,\,,\,\{-1,0,1\}^{p+1}\times\Z/k\Z\,).
\]
Then
\[
R_p(x\leftrightarrow\Gamma_\infty\setminus B_{\Gamma_\infty}(x,r))\ll_p\frac{1}{\deg(\Gamma_\infty)},
\]
and
\[
R_{\Gamma_0,p}\ll_p\frac{1}{\deg(\Gamma_0)}.
\]
\end{prop}
In particular, there is no way of avoiding the gap of $(\log r)^{p-1}\deg(\Gamma)$ that appears between the upper and lower bounds of those theorems in these settings.
\begin{proof}
Using Proposition \ref{prop:iso-growth}, one may check that for $A\subset\Gamma_0$ with $|A|\le\frac{1}{2}|\Gamma_0|$, or for $A\subset\Gamma_\infty$ with $|A|<\infty$, we have
\[
|\partial A|\gg_p
\begin{cases}
   k^\frac{1}{p+1}|A|^\frac{p}{p+1} & \text{if $|A|\le kr^\frac{p^2-1}{p}$},\\
   kr^{p-1} & \text{otherwise}.
\end{cases}
\]
The desired bounds therefore follow from Theorems \ref{thm:bk.inf} and \ref{thm:bk}, respectively.
\end{proof}

Define $Z_{n,k}$ to be the Cayley graph $Z_{n,k}=\Cay(\Z/n\Z,\{-k,\ldots,k\})$.
\begin{lemma}[edge isoperimetric inequality for $Z_{n,k}$]\label{lem:cyclic.edge}
Let $k,n\in\N$, and suppose that $A\subset Z_{n,k}$ satisfies $k\le|A|\le n-k$. Then $|\partial^EA|\ge\frac{1}{4}k^2-1$.
\end{lemma}
\begin{proof}
First, note that if an interval $I_m=\{m,m+1,\ldots,m+k\}\subset\Z/n\Z$ contains exactly $r$ elements of $A$ then eactly $r(k+1-r)$ of the edges between elements of $I_m$ belong to $\partial^EA$. In particular, if there exists $m$ with $|I_m\cap A|=\lfloor k/2\rfloor$ then the lemma is satisfied. Next, note that $|I_m\cap A|$ differs from $|I_{m+1}\cap A|$ by at most $1$, so if there exist $m$ such that $|I_m\cap A|<\lfloor k/2\rfloor$ and $m'$ such that $|I_{m'}\cap A|>\lfloor k/2\rfloor$ then there also exists some $m''$ such that $|I_{m''}\cap A|=\lfloor k/2\rfloor$, and the lemma is satisfied as before. Replacing $A$ with its complement if necessary, we may therefore assume that $|I_m\cap A|<\lfloor k/2\rfloor$ for every $m\in\Z/n\Z$. This means in particular that every element of $A$ has at least $k/2$ neighbours outside $A$, and so $|\partial^EA|\ge k|A|/2$ and the lemma is satisfied.
\end{proof}

\begin{example}
Combining Proposition \ref{lem:cyclic.edge} with Theorem \ref{thm:bk} gives the bound
\[
R_{Z_{n,k},p}^\frac{1}{p-1}\ll_p\frac{1}{k^\frac{1}{p-1}}+\frac{n}{k^\frac{p+1}{p-1}}+\frac{\log n}{k^\frac{2}{p-1}}.
\]
Provided $k\le n/\log n$, and using Lemma \ref{lem:^p.dist}, this translates as
\begin{equation}\label{eq:resist.Z_n,k}
R_{Z_{n,k},p}\ll_p\frac{1}{k}+\frac{n^{p-1}}{k^{p+1}}.
\end{equation}
Since $\deg(Z_{n,k})\asymp k$ and $\diam(Z_{n,k})\asymp n/k$, this implies in particular that
\[
R_{Z_{n,k}}\ll\frac{1}{\deg(Z_{n,k})}\left(1+\frac{\diam(Z_{n,k})^2}{|Z_{n,k}|}\right),
\]
which matches the lower bound of Theorem \ref{thm:mainp}. This example similarly achieves the lower bound of Theorem \ref{thm:main.unimodp.linear} with $r=\diam(Z_{n,k})-1$.
\end{example}

\begin{remark}\label{rem:edge}
The vertex isoperimetric inequality for $Z_{n,k}$ is much weaker than the edge isoperimetric inequality given by Lemma \ref{lem:cyclic.edge}. Indeed, it is easy to check that $|\partial A|\ge2k$ for every $A\subset Z_{n,k}$ with $|A|\le n-2k$, and also to check that this bound is optimal. Using this and Theorem \ref{thm:bk} to bound $R_{Z_{n,k}}$ gives only
\[
R_{Z_{n,k},p}\ll\frac{n^{p-1}}{k^p}+\frac{\log n}{k}.
\]
For $n$ much larger than $k$ this is weaker by a factor of $k\asymp\deg(Z_{n,k})$ than the bound \eqref{eq:resist.Z_n,k} obtained using the edge isoperimetric inequality and Theorem \ref{thm:bk}.
\end{remark}

\section{Locality of the escape probability}\label{sec:locality}
\begin{proof}[Proof of \cref{cor:local}]
We may assume without loss of generality that $\Gamma_n$ has degree equal to that of $\Gamma$. The graphs $\Gamma_n$ have superquadratic growth by hypothesis, so \cref{thm:iso.classical} and \cite[Corollary 6.32]{ly-per} imply that $p_t(x,y)\ll_{\deg(\Gamma)}t^{-3/2}$ for all $x,y\in\Gamma_n$ and $t\in\N$ for all $n\in\N$. In particular, for all  $n\in\N$, given two vertices $x,y\in\Gamma_n$, the probability that the random walk starting at $x$ hits $y$ after time $t$ is at most $Ct^{-1/2}$, with $C>0$ depending only on $\deg(\Gamma)$. In particular, if $d(x,y)=r$ then the probability that the random walk starting at $y$ hits $x$ at least once is at most $Cr^{-1/2}$, and so
\[
\Prob[\,x\to S_{\Gamma_n}(x,r)\,]-e(\Gamma_n)\le Cr^{-1/2}.
\]
Since for each $r$ we have $\Prob[\,x\to S_{\Gamma_n}(x,r)\,]=\Prob[\,x'\to S_\Gamma(x',r)\,]$ for all $n$, it follows that $e(\Gamma_n)\to e(\Gamma)$ as required.
\end{proof}

 \footnotesize{
  \bibliographystyle{abbrv}
  \bibliography{biblio.bib}
  }
\end{document}